\documentclass[a4paper,11pt]{amsart}
\usepackage{amsmath,amssymb,url}
\usepackage[active]{srcltx}
\usepackage[T1]{fontenc}
\usepackage[latin1]{inputenc}
\usepackage{xcolor}
\usepackage{hyperref}
\usepackage{latexsym,amssymb}
\usepackage{stmaryrd}
\usepackage{accents}
\usepackage{textcomp}
\usepackage{enumerate}
\usepackage{stackrel}
\usepackage[all]{xy}
\usepackage{mathtools}
\usepackage{scalerel,stackengine}
\stackMath

\newcommand\reallywidecheck[1]{%
\savestack{\tmpbox}{\stretchto{%
  \scaleto{%
    \scalerel*[\widthof{\ensuremath{#1}}]{\kern-.6pt\bigwedge\kern-.6pt}%
    {\rule[-\textheight/2]{1ex}{\textheight}}
  }{\textheight}%
}{0.5ex}}%
\stackon[1pt]{#1}{\scalebox{-1}{\tmpbox}}%
}

\numberwithin{equation}{section}

\renewcommand{\ge}{\geqslant}
\renewcommand{\le}{\leqslant}
\newdir^{((}{{}*!/-5pt/\dir_{(}}
\newdir^{(((}{{}*!/-5pt/\dir^{(}}

\def\pv#1{\ensuremath{\mathsf{#1}}}
\def\Om#1#2{\ensuremath{\overline{\Omega}_{#1}{\pv{#2}}}}
\let\cal=\mathcal
\def\Cl#1{\ensuremath{\cal{#1}}}

\theoremstyle{plain}
\newtheorem{Thm}{Theorem}[section]
\newtheorem{Prop}[Thm]{Proposition}
\newtheorem{Lemma}[Thm]{Lemma}

\newtheorem{Cor}[Thm]{Corollary}

\theoremstyle{definition}

\newtheorem{eg}[Thm]{Example}

\begin{document}

\title{Profinite congruences and unary algebras}
\thanks{The first author acknowledges partial support by CMUP
  (UID/MAT/00144/2019) which is funded by FCT (Portugal) with
  national (MATT'S) and European structural funds (FEDER) under the
  partnership agreement PT2020. %
  The work was largely carried out at Masaryk University, whose
  hospitality is gratefully acknowledged, with the support of the FCT
  sabbatical scholarship SFRH/BSAB/142872/2018. %
  The second author was supported by Grant 19-12790S of the Grant
  Agency of the Czech Republic.}

\author[J. Almeida]{Jorge Almeida}%
\address{CMUP, Dep.\ Matem\'atica, Faculdade de Ci\^encias,
  Universidade do Porto, Rua do Campo Alegre 687, 4169-007 Porto,
  Portugal}
\email{jalmeida@fc.up.pt}

\author[O. Kl\'ima]{Ond\v rej Kl\'ima}%
\address{Dept.\ of Mathematics and Statistics, Masaryk University,
  Kotl\'a\v rsk\'a 2, 611 37 Brno, Czech Republic}%
\email{klima@math.muni.cz}


\makeatletter
\@namedef{subjclassname@2010}{%
  \textup{2010} Mathematics Subject Classification}
\makeatother

\begin{abstract}
  Profinite congruences on profinite algebras determining profinite
  quotients are difficult to describe. In particular, no constructive
  description is known of the least profinite congruence containing a
  given binary relation on the algebra. On the other hand, closed
  congruences and fully invariant congruences can be described
  constructively. In a previous paper, we conjectured that fully
  invariant closed congruences on a relatively free profinite algebra
  are always profinite. Here, we show that our conjecture fails for
  unary algebras and that closed congruences on relatively free
  profinite semigroups are not necessarily profinite. As part of our
  study of unary algebras, we establish an adjunction between
  profinite unary algebras and profinite monoids. We also show that
  the Polish representation of the free profinite unary algebra is
  faithful.
\end{abstract}

\maketitle

\section{Introduction}
\label{sec:into}

Our first concern in this paper is to try to understand congruences on
profinite algebras such that the corresponding quotients are also
profinite. Following \cite{Rhodes&Steinberg:2009qt}, we call such
congruences \emph{profinite}. For instance, if one is interested in
considering profinite presentations of algebras $\langle A;
R\rangle$ (as in~\cite{Lubotzky:2001,Almeida&ACosta:2013}), one should
take on the free profinite algebra on the generating set $A$ the least
profinite congruence $\rho$ containing the relation~$R$. As in the
theory of presentations of discrete algebras, it would be useful to
have some sort of constructive description of~$\rho$. Other than
specially well-behaved cases, such as in group theory, no such
description is known in general. In all the exceptions found so far,
the reason why such a description is known is that, unlike what
happens in general, closed congruences are profinite. In the realm of
semigroup theory, many favorable examples are presented
in~\cite{Almeida&Klima:2017a}, where the aim is to find a general
method to determine when a given pseudoidentity is a consequence, in all
finite models, of a given set of pseudoidentities. Such examples are
of special type in the sense that fully invariant closed congruences
on relatively free profinite semigroups are concerned. They prompted
the authors to ask whether such congruences are always profinite even
in a more general algebraic context.

Three types of examples are given in this paper. First, an easy
example of a closed congruence on a relatively profinite semigroup
that is not profinite. Second, a modification of the first example to
make it a fully invariant closed congruence that turns out to be a
profinite congruence on a finitely generated relatively free profinite
semigroup whose quotient is countable and contains infinitely many
idempotents. Third, an example, which is again derived from the first
example, of a fully invariant closed congruence on a relatively
profinite unary algebra which is not profinite, thereby giving a
negative answer to the general case of our question raised
in~\cite{Almeida&Klima:2017a}. The latter example prompts a study of
the relationship between profinite monoids and profinite unary
algebras which leads us to exhibit a pair of adjoint functors between
the two categories.

In the general framework of universal algebra, the standard Polish
notation for terms induces a representation of the term algebra in a
free monoid, which is faithful. In the profinite setting, this leads
to a representation of the absolutely free profinite algebra in a
suitable free profinite monoid. We show that, over a nonempty
topological space of generators, such a representation is faithful if
and only if the signature is at most unary, that is, it contains no
operation symbols of arity greater than~1.

From a certain point of view, our main results and constructions
involve representations of profinite unary algebras by
appropriate profinite monoids. Each such profinite monoid appears
as the set of implicit unary operations on a given unary algebra
with a continuous multiplication given by composition. If one
would generate a clone by these operations, each new $n$-ary
operation may be identified with a unary operation of one of the
arguments. Thus, in somewhat imprecise terms, we may say that
this paper is also about profinite clones of profinite unary
algebras. However, we did not introduce the notion of profinite
clones formally, since profinite monoids are well-known and a
more suitable notion here. Notice also that in the case of a
clone on a non-unary (profinite) algebra, composition may not be
simplified into a binary operation and the behavior is much less
clear. In particular, we obtained only partial observations in
this line of investigation.

\section{Preliminaries}
\label{sec:prelims}

In this paper, we consider signatures to be sets
$\Sigma=\biguplus_{n\ge0}\Sigma_n$ graded by the natural numbers. The
set $\Sigma_n$ consists of the $n$-ary operation symbols in the
signature. The signature $\Sigma$ is a \emph{topological signature} if
each $\Sigma_n$ is a topological space; $\Sigma$ is a \emph{discrete
  signature} if each $\Sigma_n$ is a discrete space.

An \emph{algebra} of the signature $\Sigma$ is a nonempty set $S$
together with an \emph{evaluation} mapping $E_n:\Sigma_n\times S^n\to
S$ for each arity $n$; the image of the $(n+1)$-tuple
$(w,s_1,\ldots,s_n)$ is usually denoted $w_S(s_1,\ldots,s_n)$. In case
the signature is a topological signature, by a \emph{topological
  algebra} we mean an algebra $S$ endowed with a topology such that
the evaluation mappings $E_n$ are continuous. In case the topology of
the algebra is compact, we also say that the topological algebra is a
\emph{compact algebra}. All compact spaces in this paper are assumed
to be Hausdorff.

A \emph{term} is an element of an absolutely free algebra
$T_\Sigma(X)$ on a set $X$ (of \emph{variables}), which may be viewed
as an expression in the variables using operation symbols as formal
operations according to their arities. In case $X=\{x\}$, we write
$T_\Sigma(x)$ instead of $T_\Sigma(X)$. %
One often writes $\mathbf{t}=\mathbf{t}(x_1,\ldots,x_n)$ to represent
a term which involves, at most, the variables $x_1,\ldots,x_n$. For an
algebra $S$, one may define recursively for each term 
$\mathbf{t}(x_1,\ldots,x_n)$ an operation $\mathbf{t}_S:S^n\to S$ as
follows: in case $\mathbf{t}=x_i$, then $\mathbf{t}_S$ is the
projection on the $i$th component; in case
$\mathbf{t}=w(\mathbf{t}_1,\ldots,\mathbf{t}_r)$ for an operation
symbol $w$ of arity $r$, we put
$$\mathbf{t}_S(s_1,\ldots,s_n)
=w_S(\mathbf{t_1}_S(s_1,\ldots,s_n),\ldots,\mathbf{t_r}_S(s_1,\ldots,s_n)).$$

By a \emph{pseudovariety} we mean a class of finite algebras of a
fixed signature that is closed under taking homomorphic images,
subalgebras and finite direct products. For a pseudovariety \pv V, a
\emph{pro-\pv V algebra} is a compact algebra $S$ which is
residually~\pv V in the sense that, given any two points $s,s'\in S$,
there is a continuous homomorphism $\varphi:S\to T$ such that $T\in\pv
V$ and $\varphi(s)\ne\varphi(s')$. Let $\pv{Fin}_\Sigma$ denote the
pseudovariety of all finite $\Sigma$-algebras. A pro-$\pv{Fin}_\Sigma$
algebra is simply called a \emph{profinite algebra}.
Note that every profinite algebra is zero-dimensional.

In general, for two topological spaces $X$ and $T$, we denote $T^X$
the set of all functions from $X$ to $T$ and $\Cl C(X,T)$ its subset
consisting of the continuous functions. In particular, if $X$ is a
discrete space, then $\Cl C(X,T)=T^X$. Later we recall also the usual
topologies on these function spaces.

Given a topological space $X$ and a pseudovariety \pv V, there is a
\emph{free pro-\pv V algebra over $X$}, denoted \Om XV; it comes
endowed with a continuous mapping $\iota:X\to\Om XV$ (if various
pseudovarieties are present, we also denote it $\iota_{\pv V}$) and it
is characterized by the following universal property: for every
continuous mapping $\varphi:X\to S$ into a pro-\pv V algebra $S$,
there is a unique continuous homomorphism $\hat{\varphi}:\Om XV\to S$
such that the following diagram commutes:
\begin{equation}
  \label{eq:universal-free-pro-V}
  \xymatrix{
    X \ar[r]^\iota \ar[rd]_\varphi
    &
    \Om XV \ar@{-->}[d]^{\hat{\varphi}} \\
    &
    **[r]{S.}
  }
\end{equation}

In the case of a pseudovariety of semigroups \pv V, since elements of
free semigroups are usually called \emph{words}, we may refer to
elements of~\Om XV as \emph{\pv V-pseudowords on $X$} or, simply,
\emph{pseudowords}. The elements of~\Om XV that are not in the
subalgebra generated by~$X$ are said to be \emph{infinite
  pseudowords}.

There is a natural interpretation of elements of~\Om XV as (partial)
operations on a pro-\pv V algebra $S$: for each %
$u\in\Om XV$, $u_S:\Cl C(X,S)\to S$ maps each continuous function
$\varphi:X\to S$ to $\hat{\varphi}(u)$, where $\hat{\varphi}$ is the
unique continuous homomorphism of Diagram
\eqref{eq:universal-free-pro-V}. If $X$ is a finite discrete space,
which is the case that concerns us most in this paper, then $\Cl
C(X,S)=S^X$ and $u_S$ is a total operation of arity $|X|$. The
following continuity result generalizes various statements appearing
in the literature (see \cite[Subsection~2.3]{Almeida:1999c} and
\cite[Proposition~4.7]{Almeida:2003cshort}) but does not seem to have
been stated previously in its present generality. Although it may be
considered folklore, the proof is presented here for the sake of
completeness.

\begin{Thm}
  \label{t:evaluation-continuity}
  For a finite set $X$ and a pro-\pv V algebra $S$, the mapping
  \begin{align*}
    \varepsilon_{\pv V}^S:\Om XV\times S^X &\to S \\
    (u,\varphi) &\mapsto \hat{\varphi}(u)
  \end{align*}
  is continuous.
\end{Thm}

\begin{proof}
  Let \pv W be a pseudovariety contained in~\pv V and denote by %
  $\pi_{\pv W}=\widehat{\iota_{\pv W}}$ the unique continuous
  homomorphism $\Om XV\to\Om XW$ such that the diagram
  \begin{displaymath}
    \xymatrix{
      X \ar[r]^{\iota_{\pv V}} \ar[rd]_{\iota_{\pv W}}
      & \Om XV \ar[d]^{\pi_{\pv W}} \\
      & \Om XW
    }
  \end{displaymath}
  commutes. Let $h:S\to T$ be a continuous homomorphism where $S$ and
  $T$ are, respectively pro-\pv V and pro-\pv W algebras. The mapping
  $h$ induces the following diagram:
  \begin{equation}
    \label{eq:evaluation-diagram}
    \xymatrix@C=15mm{
      \Om XV\times S^X
      \ar[r]^(.6){\varepsilon_{\pv V}^S}
      \ar[d]^{\pi_{\pv W}\times(h\circ\_)} &
      S
      \ar[d]^h
      \\
      \Om XW\times T^X
      \ar[r]^(.6){\varepsilon_{\pv W}^T}
      &
      **[r]{T.}
    }
  \end{equation}
  We claim that it commutes. Indeed, given $\varphi\in S^X$, we may
  build the following diagram:
  \begin{displaymath}
    \xymatrix@R=4mm@C=4mm{
      \Om XV
      \ar[rr]^{\pi_{\pv W}}
      \ar[dd]_{\hat{\varphi}}
      &&
      \Om XW
      \ar[dd]^{\widehat{h\circ\varphi}}
      \\
      &
      X
      \ar[lu]^{\iota_{\pv V}}
      \ar[ld]^\varphi
      \ar[ru]_{\iota_{\pv W}}
      &\\
      S
      \ar[rr]_h
      &&
      **[r]{T.}
    }
  \end{displaymath}
  The left, upper, and lower triangles commute by the definition,
  respectively, of $\hat{\varphi}$, $\pi_{\pv W}$, and
  $\widehat{h\circ\varphi}$. It follows that the whole diagram
  commutes. Hence, given $u\in\Om XV$, we obtain the following
  equalities
  \begin{displaymath}
    h\bigl(\varepsilon_{\pv V}^S(u,\varphi)\bigr)
    =h\bigl(\hat{\varphi}(u)\bigr)
    =\widehat{h\circ\varphi}\bigl(\pi_{\pv W}(u)\bigr)
    =\varepsilon_{\pv W}^T\bigl(\pi_{\pv W}(u),h\circ\varphi\bigr),
  \end{displaymath}
  thereby showing that Diagram~\eqref{eq:evaluation-diagram} commutes.

  Now, since $S$ is zero-dimensional, to prove that the mapping
  $\varepsilon_{\pv V}^S$ is continuous, it suffices to show that, for
  every clopen subset $K$ of~$S$, the set %
  $(\varepsilon_{\pv V}^S)^{-1}(K)$ is open in the product space %
  $\Om XV\times S^X$. By \cite[Theorem~3.6.1]{Almeida:1994a}, there is
  a continuous homomorphism $h:S\to T$ to an algebra $T$ from~\pv V
  such that $K=h^{-1}\bigl(h(K)\bigr)$. Let \pv W be the pseudovariety
  generated by~$T$ and consider the corresponding commutative
  Diagram~\eqref{eq:evaluation-diagram}. Then the following equalities
  hold:
  \begin{displaymath}
    (\varepsilon_{\pv V}^S)^{-1}(K)
    =(\varepsilon_{\pv V}^S)^{-1}\Bigl(h^{-1}\bigl(h(K)\bigr)\Bigr)
    =\bigl(\pi_{\pv W}\times(h\circ\_)\bigr)^{-1}
    \Bigl((\varepsilon_{\pv W}^T)^{-1}\bigl(h(K)\bigr)\Bigr).
  \end{displaymath}
  Since $\Om XW\times T^X$ is a finite discrete space, the continuity
  of $\varepsilon_{\pv V}^S$ reduces to that of the mapping %
  $h\circ\_ :S^X\to T^X$, which is just $h$ on
  each component, whence continuous.
\end{proof}

The following are immediate consequences of
Theorem~\ref{t:evaluation-continuity} that we record here for later
reference.

\begin{Cor} 
  \label{c:implops-continuous}
  Let $S$ be a pro-\pv V algebra and $X$ a finite set. Then, for every
  $w\in \Om XV$, the mapping $w_S:S^X\to S$ defined by
  $w_S(\varphi)=\hat{\varphi}(w)$ is continuous.\qed
\end{Cor}

\begin{Cor}
  \label{c:implops}
  Let $h :S\to T$ be a continuous homomorphism between pro-\pv V
  algebras and $X$ a finite set. Then the following diagram commutes
  for each $w\in\Om XV$:
  \begin{displaymath}
    \xymatrix{
      S^X \ar[r]^{w_S} \ar[d]^{h\circ\_}
      &
      S \ar[d]^h
      \\
      T^X \ar[r]^{w_T}
      &
      **[r]{T.}
    }
  \end{displaymath}
\end{Cor}

\begin{proof}
  This follows from the commutativity of
  Diagram~\eqref{eq:evaluation-diagram}
  where we consider $\pv W=\pv V$ and $\pi_{\pv W}$ is 
  the identity mapping.
\end{proof}

In particular, a pro-\pv V algebra may be viewed as a natural
topological $\Sigma$-algebra for the signature
$\Sigma=\bigcup_{n=1}^\infty\Om{X_n}V$, where $\{x_1,x_2,\ldots\}$ is
a countable set of distinct variables and $X_n=\{x_1,\ldots,x_n\}$.
Note that indeed the value of an $n$-ary operation
depends continuously jointly on the $n$ arguments and the operation
itself. In particular, taking the pro-\pv V algebra to be any of the
free pro-\pv V algebras \Om{X_n}V, we see that $\Sigma$ is a clone.

An equivalence relation on a topological space $X$ is said to be
\emph{closed} (respectively \emph{clopen}) if it is a subset of
$X\times X$ with that property. In particular, it is well known that
the equality relation $\Delta_X$ on $X$ is closed if and only if $X$
is Hausdorff. A congruence $\theta$ on a topological algebra $S$ is
\emph{profinite} if $S/\theta$ is a profinite algebra for the quotient
algebraic and topological structures.

Sometimes, it is useful to consider presentations in the profinite
context. Given a pseudovariety \pv V, a set $X$, and a binary relation
$R\subseteq(\Om XV)^2$, the pro-\pv V algebra \emph{presented} by
$\langle X; R\rangle$ is the largest quotient $\Om XV/\theta$ where
$\theta$ is a profinite congruence containing~$R$. So, we are
interested in determining the smallest profinite congruence on \Om XV
that contains $R$; note that it is the intersection of all clopen
congruences containing $R$.
    
On the other hand, one may describe ``constructively'' the smallest
closed congruence on~\Om XV containing $R$ by applying the following
natural ``procedure'' considered in~\cite{Almeida&Klima:2017a}:
\begin{itemize}
\item start with the set
  \begin{align*}
    \Delta_{\Om XV}\cup
    &\left\{\bigl(\mathbf{t}_{\Om XV}(a,b_1,\ldots,b_n), %
      \mathbf{t}_{\Om XV}(a',b_1,\ldots,b_n)\bigr):\right.\\
    &\qquad\mathbf{t}(x_1,\ldots,x_{n+1}) \text{ a term},
      (a,a')\in R\cup R^{-1}, b_i\in\Om XV\Big\};
  \end{align*}
\item  alternate successively transitive closure
  and topological closure;
\item repeat transfinitely until the relation obtained stabilizes.
\end{itemize}
If a fully invariant congruence is sought, it suffices to replace the
occurrences of $a$ and $a'$ in the evaluation of the term operations
in the above procedure by $\varphi(a)$ and $\varphi(a')$,
respectively, where $\varphi$ runs over the monoid of continuous
endomorphisms of~\Om XV. As argued in~\cite{Almeida&Klima:2017a}, the
profiniteness of the smallest fully invariant closed congruence
containing $R$ may be viewed as a sort of completeness theorem in the
proof theory for pseudoidentities. Unlike the expectation expressed in
that paper, we show in this paper that such a completeness theorem
does not hold in full generality.

It should be noted that whether a closed congruence on a profinite
algebra is profinite is a purely topological property. Indeed,
Gehrke~\cite[Theorem~4.3]{Gehrke:2016a} proved that a quotient of a
profinite algebra is profinite if and only if it is zero-dimensional.

For a profinite semigroup $S$, we denote $S^I$ the semigroup obtained
from $S$ by adding a new identity element, even if $S$ already has
one, which is an isolated point. Note that $S^I$ is again a profinite
semigroup. Moreover, if $\varphi:S\to T$ is a continuous homomorphism
between profinite semigroups, then so is its extension %
$\varphi^I:S^I\to T^I$ sending the new identity element of $S^I$ to
the new identity element of~$T^I$.

Given an element $s$ of a profinite semigroup $S$, the sequence
$(s^{n!})_n$ converges to an idempotent, denoted $s^\omega$; this
follows from the simple combinatorial fact that, if the semigroup is
finite, then the sequence in question is eventually constant with
idempotent value.

For a positive integer $n$,
the pseudovariety $\pv K_n$ consists of all finite semigroups
satisfying the identity $x_1\cdots x_ny=x_1\cdots x_n$. The union of
the ascending chain of pseudovarieties $\pv K_n$ is denoted \pv K. For
a finite set $A$, the structure of both $\Om AK_n$ and \Om AK is well
known and quite transparent: the former consists of all words in the
letters of~$A$ of length at most $n$ where the product of two such
words is the longest prefix of their concatenation of length at
most~$n$; the latter is realized as the set of all finite words
$a_1\ldots a_m$ together with all right infinite words $a_1a_2\ldots$,
with each $a_i$ in~$A$, where right infinite words are left zeros and
otherwise multiplication is obtained by concatenation. By the
\emph{content} of a finite or infinite word we mean the set of letters
that appear in it.

\section{A non-profinite closed congruence}
\label{sec:non-profinite-closed}

Our starting observation is the well-known fact that there is a
continuous mapping $C\to[0,1]$ from the Cantor set $C$ onto the
interval $[0,1]$ of real numbers. One simple description of such a
mapping $\varphi$ is obtained by first writing each element $s$ of the
Cantor set as a ternary expansion $s=\sum_{n=1}^\infty a_n/3^n$ with
$a_n\in\{0,2\}$, which is identified with the infinite word
$a_1a_2\cdots$; replacing the digit 2 by 1 and viewing the result as a
binary expansion $\sum_{n=1}^\infty a_n/2^{n+1}$ of a real number, we
obtain the image $\varphi(s)$. Because of the ambiguity in binary
expansions resulting from the equality $\sum_{n=1}^\infty1/2^n=1$, we
see that $\varphi$ identifies two elements of $C$, viewed as infinite
words on the alphabet $\{0,2\}$, if and only if they constitute a pair
of the form $w02^\infty,w20^\infty$.

The above observations suggest considering the closed congruence
$\theta$ on the profinite semigroup \Om AK generated by the pair
$(ab^\omega,ba^\omega)$, where $A=\{a,b\}$.

\begin{Lemma}
  \label{l:theta}
  The closed congruence $\theta$ is obtained at the first step of the
  procedure described in Section~\ref{sec:prelims}, that is, it
  consists of the diagonal $\Delta_{\Om AK}$ plus all pairs of the
  form $(wab^\omega,wba^\omega)$ together with
  $(wba^\omega,wab^\omega)$, with $w\in A^*$.
\end{Lemma}

\begin{proof}
  It suffices to show that the relation described in the statement of
  the lemma is a closed congruence. It is clearly closed and a
  reflexive and symmetric binary relation that is stable under both
  left and right multiplication by the same element. Since each
  element is related with at most one element different from it, the
  relation is also transitive.
\end{proof}

The following lemma captures a key property that holds in finite
quotients of $\Om AK/\theta$.

\begin{Lemma}
  \label{l:K-closed-cong}
  If $S\in\pv K_n$ and $s,t$ are elements of $S$ such that
  $st^\omega=ts^\omega$, then the subsemigroup $\langle s,t\rangle$ is
  nilpotent.
\end{Lemma}

\begin{proof}
  To prove the lemma, it suffices to establish the claim that
  $usv=utv$ whenever $u,v$ are products of factors $s$ and $t$ with a
  total number of factors for $uv$ equal to $n-1$. To prove it, we
  proceed by induction on the number $k$ of factors in~$v$. If $k=0$,
  then we have $us=ust^\omega=uts^\omega=ut$. Assuming that $v=sv'$ and that
  the claim holds for right factors which are products of $k-1$ factors $s$ and $t$,
  we obtain
  \begin{displaymath}
    usv %
    =ussv'%
    =usst^\omega %
    =usts^\omega %
    =ustt^\omega %
    =uts^\omega %
    =utsv'=utv
  \end{displaymath}
  and similarly for $v=tv'$.
\end{proof}

\begin{Cor}
  \label{c:K-closed-cong-1}
  If $S\in\pv K$, $s,t\in S$, and $st^\omega=ts^\omega$, then the
  equality $s^\omega=t^\omega$ holds.\qed
\end{Cor}

We may now easily obtain the following result without referring to the
topology of the real numbers.

\begin{Cor}
  \label{c:K-closed-cong-2}
  The quotient topological semigroup $\Om AK/\theta$ is not profinite.
\end{Cor}

\begin{proof}
  If it were profinite, since $a^\omega$ and $b^\omega$ are not
  $\theta$-equivalent by Lemma~\ref{l:theta}, there would be a
  continuous homomorphism $\psi:\Om AK\to S$ onto a semigroup from~\pv
  K such that $\psi(ab^\omega)=\psi(ba^\omega)$ and
  $\psi(a^\omega)\ne\psi(b^\omega)$. This is impossible by
  Corollary~\ref{c:K-closed-cong-1}.
\end{proof}

Note that the minimum ideal $I$ of~\Om AK is homeomorphic with the
Cantor set and carries a multiplication in which every element is a
left-zero. By the considerations at the beginning of the section, the
restriction $\theta'$ of the congruence $\theta$ to~$I$ is such that
the quotient topological semigroup $I/\theta'$ is homeomorphic with the
unit real interval $[0,1]$, carrying also the left-zero
multiplication. Of course, $\theta'$ is an easier example of a closed
congruence on a profinite semigroup that is not profinite as no
algebraic considerations are necessary to justify this statement. Yet,
the example in this section has two advantages that justify
considering it: the profinite semigroup \Om AK is finitely generated
and the arguments are essentially algebraic.

\section{A finitely generated countable profinite semigroup with
  infinitely many idempotents}
\label{sec:K}

In~\cite{Almeida&Klima:2019b}, we have investigated in the realm of
semigroup theory the property of a pseudovariety \pv V which consists
in all finitely generated pro-\pv V algebras to be countable; such a
pseudovariety \pv V is said to be \emph{locally countable}. There, we
have shown that countable finitely generated profinite semigroups with
finitely many idempotents are rather well behaved. In particular, they
are algebraically generated by the given finite set of generators
together with the idempotents. Thus, it is natural to ask whether the
property of having only finitely many idempotents always holds for a
countable finitely generated profinite semigroup. In this section, we
present an example of a locally countable pseudovariety whose
relatively free profinite semigroups on at least two generators have
infinitely many idempotents.

The example in this section came about in an attempt to find a fully
invariant closed congruence on a relatively free profinite semigroup
that is not a profinite congruence; the existence of such an example
remains an open problem. The starting point is the congruence of
Section~\ref{sec:non-profinite-closed} but we need to suitably modify
it to obtain a fully invariant closed congruence. While the example
turns out not to be a profinite congruence, it still enjoys properties
that justify considering it here. It also illustrates how difficult it
seems to address the conjecture of~\cite{Almeida&Klima:2017a} for
semigroups.

Let $A=\{a,b\}$ and let $\theta$ be the fully invariant closed
congruence on the profinite semigroup $\Om AK$ generated by the set of
all pairs of the form
\begin{equation}
  \label{eq:1example-K-def}
  (ab^nab^\omega,ab^{n+1}a^\omega)\quad(n\ge1).
\end{equation}

For a pseudword $w\in\Om AK$, denote by $\tilde{0}(w)$ the shortest
prefix of~$w$ with the same content as~$w$. Note that $\tilde{0}(w)$
is a well-defined finite word. it turns out to define an invariant for
$\theta$-classes.

\begin{Lemma}
  \label{l:example-K-0bar0}
  If $u\mathrel\theta v$ then $\tilde{0}(u)=\tilde{0}(v)$.
\end{Lemma}

\begin{proof}
  The lemma is proved by transfinite induction on the construction
  of~$\theta$. In step 0, we consider the defining
  pairs~\eqref{eq:1example-K-def}, which do have the required property
  since both components start by $ab$; this property is clearly
  preserved by substitution and multiplication on the left;
  multiplication on the right does not affect it since all nontrivial
  pairs have left-zero components. The property is also clearly
  preserved in the transitive closure steps. Finally, it is also
  preserved in the topological closure steps since the languages of
  the form $wA^*$ are \pv K-recognizable, which implies that the
  prefix of a fixed length defines a continuous function on~\Om AK.
\end{proof}

Next, we establish that $\theta$ identifies various pairs of infinite
pseudowords.

\begin{Lemma}
  \label{l:example-K}
  The following $\theta$ relations hold:
  \begin{enumerate}[(i)]
  \item\label{item:example-K-1}
    $ab^na^\omega\mathrel\theta ab^\omega$ for $n\ge3$;
  \item\label{item:example-K-2}
    $ab^\omega\mathrel\theta (ab^n)^\omega$ for $n\ge2$;
  \item\label{item:example-K-3}
    $(ab)^2a^\omega\mathrel\theta (ab)^\omega$;
  \item\label{item:example-K-4}
    $(ab)^\omega\mathrel\theta (aba)^\omega$;
  \item\label{item:example-K-5}
    $(aba)^\omega\mathrel\theta aba^\omega$;
  \item\label{item:example-K-6}
    $(ab)^\omega\mathrel\theta aba^\omega$;
  \item\label{item:example-K-7}
    $ab^\omega\mathrel\theta aba^\omega$;
  \item\label{item:example-K-8}
    $ab^\omega\mathrel\theta abw$ for every $w\in\Om AK\setminus A^+$.
  \end{enumerate}
\end{Lemma}

\begin{proof}
  To prove~(\ref{item:example-K-1}), respectively from the definition
  of~$\theta$ and its full invariance, we observe that
  \begin{displaymath}
    ab^na^\omega 
    \mathrel\theta ab^{n-1}ab^\omega 
    =ab^{n-1}a(b^{n-1})^\omega 
    \mathrel\theta ab^{2(n-1)}a^\omega.
  \end{displaymath}
  Hence, for every $n\ge3$, there is $m>n$ such that %
  $ab^na^\omega\mathrel\theta ab^ma^\omega$. Iterating this procedure,
  we obtain a strictly increasing sequence $(n_k)_k$ of integers
  starting at $n$ such that all $ab^{n_k}a^\omega$ lie in the same
  $\theta$-class. Since $\theta$ is closed, it follows that
  $ab^\omega=\lim ab^{n_k}a^\omega$ also lies in the same
  $\theta$-class.

  For~(\ref{item:example-K-2}), note that, for $n\ge 2$, 
  \begin{displaymath}
    ab^\omega 
    \stackrel[\text{(\ref{item:example-K-1})}]{}\theta ab^{2n}a^\omega 
    =ab^nb^na^\omega 
    \mathrel\theta ab^na(b^n)^\omega.
  \end{displaymath}

  Hence, we have $ab^\omega\mathrel\theta ab^nab^\omega$. Since
  $\theta$ is stable under multiplication and transitive, it follows
  that $ab^\omega\mathrel\theta (ab^n)^mab^\omega$ for every $m\ge1$.
  Taking limits, we conclude that %
  $ab^\omega\mathrel\theta (ab^n)^\omega$.

  The relation~(\ref{item:example-K-3}) follows from the following
  calculations:
  \begin{displaymath}
    (ab)^2a^\omega 
    =a\cdot bab\cdot a^\omega 
    \stackrel[\text{(\ref{item:example-K-2})}]{}\theta a(baba^2)^\omega 
    =(a(ba)^2)^\omega 
    \stackrel[\text{(\ref{item:example-K-2})}]{}\theta a(ba)^\omega 
    =(ab)^\omega.
  \end{displaymath}

  For~(\ref{item:example-K-4}), note that
  \begin{displaymath}
    aba\cdot(ab)^\omega
    =aba^2(ba)^\omega
    =a\cdot ba\cdot a\cdot(ba)^\omega
    \mathrel\theta a(ba)^2 a^\omega
    =(ab)^2a^\omega
    \stackrel[\text{(\ref{item:example-K-3})}]{}\theta (ab)^\omega.
  \end{displaymath}
  Hence, we have $(ab)^\omega\mathrel\theta (aba)^n(ab)^\omega$ for
  every $n\ge1$. Taking limits, yields~(\ref{item:example-K-4}).

  Next, for~(\ref{item:example-K-5}), we observe that
  \begin{displaymath}
    aba^\omega
    \stackrel[\text{(\ref{item:example-K-1})}]{}\theta aba^3b^\omega
    \mathrel\theta aba^2ba^\omega
    =aba\cdot aba^\omega.
  \end{displaymath}
  Iterating, we get $aba^\omega\mathrel\theta (aba)^naba^\omega$
  which, taking limits, yields ~(\ref{item:example-K-5}).

  Transitivity applied to~(\ref{item:example-K-4})
  and~(\ref{item:example-K-5}) gives~(\ref{item:example-K-6}).

  For~(\ref{item:example-K-7}), we first note that
  \begin{displaymath}
    ab^\omega 
    \stackrel[\text{(\ref{item:example-K-2})}]{}\theta (ab^2)^\omega
    \stackrel[\text{(\ref{item:example-K-6})}]{}\theta ab^2(ab)^\omega
    \stackrel[\text{(\ref{item:example-K-6})}]{}\theta ab^2aba^\omega
    \mathrel\theta ab^2a^2b^\omega
    =ab^2a\cdot ab^\omega
  \end{displaymath}
  where, in the second step we use invariance of~$\theta$ under the
  substitution sending $a$ to $ab$ and fixing~$b$. Iterating and
  taking limits yields the first step in the following string of
  relations:
  \begin{displaymath}
    ab^\omega
    \mathrel\theta (ab^2a)^\omega
    =a(b^2a^2)^\omega
    \stackrel[\text{(\ref{item:example-K-2})}]{}\theta ab^2a^\omega
    \mathrel\theta ab\cdot ab^\omega.
  \end{displaymath}
  Again, iterating and taking limits gives %
  $ab^\omega\mathrel\theta (ab)^\omega$ which, combined
  with~(\ref{item:example-K-6}) yields~(\ref{item:example-K-7}).

  Finally, for~(\ref{item:example-K-8}), we start by observing that,
  substituting $b$ by $b^m$ in~(\ref{item:example-K-7}), we get
  $ab^\omega\mathrel\theta ab^ma^\omega$. Interchanging $a$ and $b$
  and applying the same argument to the suffix $ba^\omega$ of
  $ab^ma^\omega$ yields $ab^\omega\mathrel\theta ab^ma^nb^\omega$.
  Thus, multiplying on the left any element of the $\theta$-class of
  $ab^\omega$ by $ab^ma^{n-1}$ with $m,n\ge1$ we stay in the same
  class. We may, therefore obtain any infinite word in the letters $a$
  and $b$ starting with $ab$, which proves~(\ref{item:example-K-8}).
\end{proof}

Note that in the proof of
Lemma~\ref{l:example-K}(\ref{item:example-K-8}) we only used the
relation $ab^\omega\mathrel\theta aba^\omega$. Combined with
Lemma~\ref{l:example-K}, this observation shows that $\theta$ is also
the fully invariant closed congruence generated by the pair
$(aba^\omega,ab^\omega)$.

\begin{Prop}
  \label{p:example-K}
  The quotient topological semigroup $\Om AK/\theta$ is profinite.
\end{Prop}

\begin{proof}
  We must show that distinct $\theta$-classes may be separated by
  clopen unions of $\theta$-classes. The $\theta$-classes of finite
  words are singleton sets. Thus, the $\theta$-class of a finite word
  is itself a clopen set separating it from any other $\theta$-class.
  On the other hand, by Lemmas~\ref{l:example-K-0bar0}
  and~\ref{l:example-K}(\ref{item:example-K-8}), two infinite
  pseudowords are $\theta$-equivalent if and only if they have the
  same value under $\tilde{0}$. Hence the $\theta$-class of an
  infinite pseudoword $w\in\Om AK$ of full content is separated from
  any other class containing infinite pseudowords by the clopen set
  $\tilde{0}(w)\,\Om AK$. The remaining classes
  are the singletons $\{a^\omega\}$ and $\{b^\omega\}$, which are
  separated by the clopen set $a\,\Om AK$, which in turn is a union of
  $\theta$-classes by Lemma~\ref{l:example-K-0bar0}.
\end{proof}

Note that $\Om AK/\theta$ is a countable relatively free profinite
semigroup with infinitely many idempotents.

\section{A non-profinite fully invariant closed congruence}
\label{sec:unary}

We originally proposed our conjecture that fully invariant closed
congruences on relatively free profinite semigroups are profinite as
being valid in all reasonable algebraic contexts
\cite{Almeida&Klima:2017a}. We proceed to present a counterexample for
unary algebras.

Throughout the remainder of the paper, we fix a finite set $A$ whose
elements are viewed as symbols of unary operations. Note that an
\emph{$A$-algebra} is a nonempty set $U$ together with a function
$a_U:U\to U$ for each $a\in A$. These functions generate a
subsemigroup $\Cl S_U$ of the full transformation semigroup $\Cl T_U$
of the set $U$ and induce a homomorphism $\alpha_U:A^+\to\Cl S_U$, 
where $\alpha_U(a)=a_U$. 
For
$w\in A^+$ and $u\in U$, we write $w\cdot u$ for $\alpha_U(w)(u)$.
  
An onto homomorphism $\varphi:U\to V$ of $A$-algebras induces a
homomorphism $\bar\varphi:\Cl S_U\to\Cl S_V$. Indeed, for $w,w'\in
A^+$, $\alpha_U(w)=\alpha_U(w')$ means that, for every $u\in U$,
$w\cdot u=w'\cdot u$, which implies that
$w\cdot\varphi(u)=w'\cdot\varphi(u)$. Since $\varphi(u)$ represents an
arbitrary element of~$V$, we conclude that $\alpha_V(w)=\alpha_V(w')$.
Hence, for the homomorphism $\bar\varphi$, the following diagram
commutes: 
\begin{displaymath}
  \xymatrix@R=3pt{
    & **[r]{\Cl S_U}
    \ar@{-->}[dd]^{\bar\varphi} \\
    A^+
    \ar[ru]^{\alpha_U}
    \ar[rd]_{\alpha_V}
    &\\
    & **[r]{\Cl S_V.} }
\end{displaymath}
Note that $\bar\varphi$ is onto and it is the unique homomorphism for which 
the diagram commutes. From these observations one may deduce that 
$\overline{\psi\circ\varphi}=\bar{\psi}\circ \bar{\varphi}$, for 
every pair of onto homomorphisms
$\varphi : U \rightarrow V$ and $\psi : V\rightarrow W$.
This property suggests that one may interpret the operator 
$\Cl S$ applied on $A$-algebras together with the operator 
$\bar{\phantom{\varphi}}$ applied on homomorphisms between 
$A$-algebras as a functor from the category of $A$-algebras
to the category of semigroups (or monoids).

Given an $A$-generated semigroup $S$, we may consider $S^I$ as an
$A$-algebra where $a\cdot s=as$ for $a\in A$ and $s\in S^I$. Note that
$\Cl S_{S^I}\simeq S$ is essentially the Cayley representation theorem
for semigroups.
To continue with the categorical point of view, one may see the
operator $\underline{\phantom{S}}^I$ as a functor which is adjoint to
the functor mentioned in the previous paragraph. For discrete algebras
and monoids this was established in~\cite{Planting:2013}.
This idea is not needed for the purpose of this section, but we return
to it in the next two sections, where we establish appropriate results
in the realm of profinite $A$-algebras and monoids.

\begin{Lemma}
  \label{l:A-algebra-from-ts}
  Let $U$ be an $A$-algebra generated by an element $x$. Then the
  mapping $1\mapsto x$ extends uniquely to an onto homomorphism of
  $A$-algebras $(\Cl S_U)^I\to U$.
\end{Lemma}

\begin{proof}
  Given $w,w'\in A^+$ acting the same way on every element of~$U$,
  we have, in particular $w\cdot x=w'\cdot x$. Hence, we may
  define a mapping $\varphi:(\Cl S_U)^I\to U$ by
  $\varphi(\alpha_U(w))=w\cdot x$ and $\varphi(1)=x$. Clearly this
  mapping respects the action of~$A$ and it is onto.  Moreover,
  $\varphi$ is the unique homomorphism mapping $1$ to~$x$.
\end{proof}

A profinite $A$-algebra $U$ is an inverse limit of finite
$A$-algebras.
Note that $A$-algebras in general do not have finitely determined
syntactic congruences, a property that is known to entail that a
compact zero-dimensional algebra is profinite (see
\cite{Clark&Davey&Freese&Jackson:2004}). Nevertheless, one may ask
whether every compact zero-dimensional $A$-algebra is profinite. We
proceed to describe a simple counterexample.

\begin{eg}
  \label{eg:profiniteness-A-algebras}
  Take $A=\{a\}$ to be a singleton set of operation symbols. Let $U$
  be the one point compactification of the set of natural numbers,
  which is obtained by adding a point $\infty$ to which all sequences
  without bounded subsequences converge. Define a structure of
  $A$-algebra on $U$ by letting $a(n)=\max\{0,n-1\}$ for $n\ge0$ and
  $a(\infty)=\infty$. Then $U$ is a compact zero-dimensional $A$-algebra. 
  Given a continuous homomorphism $\varphi:U\to F$
  into a finite $A$-algebra $F$, there must exist $m>n$ such that
  $\varphi(m)=\varphi(n)$. Then, we have
  $\varphi(n)=\varphi(a^{m-n}(m))=a^{m-n}(\varphi(n))$ and so
  $\varphi(n)=a^{k(m-n)}(\varphi(n))=\varphi(a^{k(m-n)}(m))$ for all
  $k\ge1$. Taking $k$ sufficiently large, we conclude that
  $\varphi(n)=\varphi(0)$ which yields $\varphi(r)=\varphi(0)$ for all
  $r\le m$ as $\varphi(0)$ is a fixed point under the action of~$a$.
  Since there must be arbitrarily large $m$ in the above conditions,
  we deduce that $\varphi$ is constant on natural numbers. Hence, $U$
  is not residually finite and, therefore, it is not profinite.
\end{eg}

Let $U=\varprojlim U_i$ be an inverse limit of finite $A$-algebras
with onto connecting homomorphisms $\varphi_{ij}:U_i\to U_j$ ($i\ge
j$). There is an associated inverse system of $A$-generated finite
semigroups $\Cl S_{U_i}$ with connecting (onto) homomorphisms
$\bar\varphi_{ij}:\Cl S_{U_i}\to\Cl S_{U_j}$ ($i\ge j$). Thus, we may
consider the inverse limit $\varprojlim \Cl S_{U_i}$,
which may be described as the closed subsemigroup of the product
$\prod_i\Cl S_{U_i}$ consisting of all $(s_i)_i$ (with $s_i\in\Cl
S_{U_i}$) such that $\bar\varphi_{ij}(s_i)=s_j$ whenever $i\ge j$.

As defined above, $\Cl S_U$ is the subsemigroup of~$\Cl T_U$ generated
by the $a_U$ ($a\in A$). For each $i$, there is an onto homomorphism
of $A$-algebras $\varphi_i:U\to U_i$, which induces an onto
semigroup homomorphism $\bar\varphi_i:\Cl S_U\to\Cl S_{U_i}$.

Suppose that $w,w'\in A^+$ are such that there exists $u\in U$ with
$w\cdot u\ne w'\cdot u$. Then, there exists $i$ such that
\begin{displaymath}
  w\cdot\varphi_i(u)
  =\varphi_i(w\cdot u)
  \ne\varphi_i(w'\cdot u)
  =w'\cdot\varphi_i(u).
\end{displaymath}
Hence, the homomorphism $\Cl S_U\to\prod_i\Cl S_{U_i}$ induced by the
$\bar\varphi_i$ is injective and it takes values in the profinite
semigroup $\varprojlim \Cl S_{U_i}$, 
which we denote $\hat{\Cl S}_U$.

\begin{Lemma}
  \label{l:completion}
  Let $V=\varprojlim V_\gamma$ be another inverse limit of finite
  $A$-algebras with onto connecting homomorphisms and let $\rho:V\to
  U$ be an onto continuous homomorphism. Then there is a continuous
  semigroup homomorphism $\hat{\rho}:\hat{\Cl S}_V\to\hat{\Cl S}_U$
  respecting generators from $A$.
\end{Lemma}

Before establishing Lemma~\ref{l:completion}, we register the
following immediate application which amounts to the independence of
the profinite semigroup $\hat{\Cl S}_U$ on the particular expression
$U=\varprojlim U_i$ of the profinite $A$-algebra $U$ as an inverse
limit of finite $A$-algebras.

\begin{Cor}
  \label{c:completion}
  Up to isomorphism, the profinite semigroup $\hat{\Cl S}_U$
  defined above does not depend on the concrete inverse limit
  $U=\varprojlim U_i$ of finite $A$-algebras with onto connecting
  homomorphisms considered but only on $U$.
\end{Cor}

\begin{proof}[Proof of Lemma~\ref{l:completion}]
  Let $\psi_\gamma:V\to V_\gamma$ be the natural mapping. Let
  $\hat{\Cl S}_{U}=\varprojlim\Cl S_{U_i}$ and %
  $\hat{\Cl S}_{V}=\varprojlim\Cl S_{V_\gamma}$. For every $i$, the
  composite continuous homomorphism %
  $\varphi_i\circ\rho$ entails the
  existence of an index $\gamma_i$ and a
  homomorphism $\rho_{\gamma i}:V_{\gamma_i}\to U_i$ such that the
  following diagram commutes (see, for instance,
  \cite[Lemma~3.1.37]{Rhodes&Steinberg:2009qt}):
  \begin{displaymath}
    \xymatrix{%
      V \ar[r]^{\psi_{\gamma_i}}
      \ar[d]^\rho
      & V_{\gamma_i} \ar@{-->}[d]^{\rho_{\gamma_i}} \\
      U \ar[r]^{\varphi_i}
      & \,U_i.}
  \end{displaymath}
  It follows that there is an onto homomorphism %
  $\Cl S_{V_{\gamma_i}}\to\Cl S_{U_i}$ respecting the choice of
  generators from~$A$, whence also an onto continuous homomorphism
  $\delta_i:\hat{\Cl S}_{V}\to\Cl S_{U_i}$, which does not depend on the choice
  of~$\gamma_i$ as it also respects the choice of generators from~$A$.
  If $i\ge j$, it follows that the diagram
  \begin{displaymath}
    \xymatrix@R=3pt{
      & \Cl S_{U_i}
      \ar[dd]^{\bar\varphi_{ij}} \\
      \hat{\Cl S}_V
      \ar[ru]^{\delta_i}
      \ar[rd]_{\delta_j}
      &\\
      & \Cl S_{U_j}}
  \end{displaymath}
  commutes. From the universal property of the inverse limit, we
  deduce 
  that there exists a continuous homomorphism $\hat{\Cl
    S}_V\to\hat{\Cl S}_U$ respecting generators.
\end{proof}

We have shown how to associate an $A$-generated profinite semigroup
with a profinite $A$-algebra. We now show how to, conversely,
associate a profinite $A$-algebra with an $A$-generated profinite
semigroup.

\begin{Lemma}
  \label{l:Cayley-profinite}
  Let $S$ be an $A$-generated profinite semigroup. Then $U=S^I$ is a
  profinite $A$-algebra and $\hat{\Cl S}_U$ is isomorphic with~$S$ as
  a topological semigroup.
\end{Lemma}

\begin{proof}
  Let $S=\varprojlim S_i$ be a description of $S$ as an inverse limit
  of finite semigroups with onto connecting homomorphisms. Then, we
  have $U=S^I=\varprojlim S_i^I$ and so %
  $\hat{\Cl S}_U=\varprojlim \Cl S_{S^I_i}\simeq\varprojlim S_i=S$.
\end{proof}

Recall that $\Cl S_U\subseteq \Cl T_U$. In other words, 
there is an action of $\Cl S_U$ on $U$, namely 
the mapping $\Cl S_U \times U \rightarrow U$
which maps $(\alpha_U(w),u)$ to $w\cdot u$.
Elements of $\hat{\Cl S}_U$ may also be viewed as transformations
of~$U$ as the next result shows. 

\begin{Prop}
  \label{p:profinite-action}
  For a profinite $A$-algebra $U$, there is a natural continuous
  action of $\hat{\Cl S}_U$ on $U$, that is a continuous mapping
  $\beta:\hat{\Cl S}_U\times U\to U$ such that 
  $\beta(\alpha_U(w),u)=w\cdot u$ for every $w\in A^+$. 
  The corresponding mapping
  $\hat{\Cl S}_U\to\Cl T_U$, which maps $s$ to $\beta(s,\_)$, is injective.
\end{Prop}

For $s\in\hat{\Cl S}_U$ and $u\in U$, we denote the image under
$\beta$ of the pair $(s,u)$ also by $s\cdot u$.
  
\begin{proof}
  Let $U=\varprojlim U_i$ be an inverse limit of finite $A$-algebras
  with onto connecting homomorphisms $\varphi_{ij}$. For the natural
  continuous homomorphism $\varphi_i:U\to U_i$, we have a
  corresponding continuous homomorphism %
  $\hat\varphi_i:\hat{\Cl S}_U\to\Cl S_{U_i}$. The two mappings
  provide a continuous mapping %
  $\hat\varphi_i\times\varphi_i:\hat{\Cl S}_U\times U %
  \to\Cl S_{U_i}\times U_i $. %
  Composing with the action $\Cl S_{U_i}\times U_i\to U_i$ gives a
  continuous mapping $\psi_i:\hat{\Cl S}_U\times U\to U_i$. From the
  definition of these mappings, it is clear that, if $i\ge j$, then
  $\varphi_{ij}\circ \psi_i=\psi_j$. Hence, the mappings $\psi_i$
  induce a continuous mapping $\beta:\hat{\Cl S}_U\times U\to U$ such
  that the following diagram commutes for every $i$:
  \begin{displaymath}
    \xymatrix{
      \hat{\Cl S}_U\times U \ar@{-->}[r]^(.6)\beta
      \ar[d]^{\hat\varphi_i\times\varphi_i}
      &
      U \ar[d]^{\varphi_i} \\
      \Cl S_{U_i}\times U_i \ar[r]
      & **[r]{U_i.}
    }
  \end{displaymath}
  For $s\in \Cl S_U$, that is $s=\alpha_U(w)$ with $w\in A^+$, we have
  $\varphi_i(\beta(s,u))=w\cdot \varphi_i(u)$, and consequently 
  $\beta(s,u)=w\cdot u$.
  
  Given distinct $s,t\in\hat{\Cl S}_U$, there is $i$ such that
  $\hat\varphi_i(s)\ne\hat\varphi_i(t)$. Since $\Cl S_{U_i}$ is a
  subsemigroup of $\Cl T_{U_i}$, there is $u_i\in U_i$ such that
  $\hat{\varphi}_i(s)(u_i)\ne\hat{\varphi}_i(t)(u_i)$. From the
  commutativity of the above diagram it follows that
  $\beta(s,u)\ne\beta(t,u)$ for every $u\in\varphi_i^{-1}(u_i)$, which
  proves the required injectivity.
\end{proof}

Notice that, for every $s\in\hat{\Cl S}_U$, the transformation 
$\beta(s,\_)$ is continuous as we show in the next section. 
This aspect is not important for the purpose of this section but plays
a role later in the paper.

Let \pv{Un} be the pseudovariety consisting of all finite $A$-algebras
and let $X=\{x\}$. Abusing notation, we may write \Om xV instead of
\Om XV for a pseudovariety \pv V of $A$-algebras.
  
\begin{Lemma}
  \label{l:iso}
  The $A$-generated profinite semigroup %
  $S=\hat{\Cl S}_{\Om x{Un}}$ is free.
\end{Lemma}

\begin{proof}
  Let $T$ be an $A$-generated finite semigroup. Consider the unique
  continuous homomorphism of $A$-algebras $\varphi:\Om x{Un}\to T^I$
  mapping $x$ to~$1$. By Lemma~\ref{l:completion}, 
  it induces an onto continuous homomorphism of
  semigroups $S\to\Cl S_{T^I}=T$ respecting the generators from~$A$.
  Hence, $S$~is freely generated by $A$ as a profinite semigroup.
\end{proof}

Denote by \emph{$\eta$} the unique continuous isomorphism %
$\Om AS\to\hat{\Cl S}_{\Om x{Un}}$ respecting generators.

Consider now a pseudovariety \pv V of semigroups. There is an
associated pseudovariety \emph{$\pv V^u$} of $A$-algebras which is
defined by the pseudoidentities of $A$-algebras $\eta(u)\cdot
x=\eta(v)\cdot x$, where $u=v$ runs over all semigroup
pseudoidentities on~$A$ satisfied by~\pv V. Note that, in particular,
$\pv{Un}=\pv S^u$. Let $\lambda:\Om x{Un}\to\Om xV^u$ and %
$\psi:\Om AS\to\Om AV$ be the natural onto homomorphisms.

By definition, $\eta$ induces a continuous homomorphism %
$\eta_\pv V:\Om AV\to\hat{\Cl S}_{\Om xV^u}$. Indeed, let %
$u,v\in\Om AS$. If $\psi(u)=\psi(v)$, then the pseudoidentity $u=v$
holds in~\pv V, so that by definition of $\pv V^u$, the pseudoidentity
$\eta(u)\cdot x=\eta(v)\cdot x$ holds in~$\pv V^u$; substituting an
arbitrary element of~\Om x{Un}
for~$x$, it follows that
$\bar\lambda(\eta(u))=\bar\lambda(\eta(v))$. Hence, there is indeed a
continuous homomorphism $\eta_{\pv V}$ such that the following diagram
commutes:
\begin{displaymath}
  \xymatrix{ 
    \Om AS \ar[r]^(.4)\eta \ar[d]^\psi 
    & \,\hat{\Cl S}_{\Om x{Un}}\, \ar[d]^{\hat\lambda} \\
    \Om AV \ar[r]^(.45){\eta_\pv V} 
    & \,\hat{\Cl S}_{\Om xV^u}.}
\end{displaymath}
The continuity of $\eta_\pv V$ follows from the compactness of \Om AS.

\begin{Prop}
  \label{p:iso2}
  There is a continuous isomorphism of profinite $A$-algebras %
  $\varphi_{\pv V}:\Om xV^u\to(\Om AV)^I$ such that, for every %
  $w\in\Om xV^u$, the equality %
  \begin{equation}
    \label{eq:iso2-1}
    w=\eta_{\pv V}^I(\varphi_{\pv V}(w))\cdot x
  \end{equation}
  holds.
\end{Prop}

In particular, $(\Om AS)^I$ is the free 1-generated profinite
$A$-algebra.
  
\begin{proof}[Proof of Proposition~\ref{p:iso2}]
  Given $u,v\in\Om AS$ such that the pseudoidentity $u=v$ is satisfied
  by~\pv V, $uw=vw$ is also valid in~\pv V for all $w\in\Om AS$. Hence
  the profinite $A$-algebra $(\Om AV)^I$ satisfies the
  defining pseudoidentities for~$\pv V^u$ and, whence, it is a
  pro-$\pv V^u$ $A$-algebra. Hence,  
  there is a continuous homomorphism of profinite $A$-algebras %
  $\varphi_{\pv V}:\Om xV^u\to(\Om AV)^I$ mapping $x$ to the identity
  element of~$(\Om AV)^I$.

  Let $w$ be an arbitrary point of the $A$-algebra $\Om xV^u$. To
  establish that the equality \eqref{eq:iso2-1} holds and, therefore,
  that $\varphi_{\pv V}$ is an isomorphism, consider the following
  diagram:
  \begin{equation}
    \label{eq:iso2-2}
    \xymatrix{
      (\Om AV)^I \ar[r]^{\eta_{\pv V}^I} \ar[rd]^\nu
      &
      (\hat{\Cl S}_{\Om xV^u})^I \ar[d]^{\cdot x}
      \\
      \Om xV^u \ar[u]^{\varphi_{\pv V}} \ar@{-->}[r]^{\mathrm{id}}
      &
      \Om xV^u
    }
  \end{equation}
  where $\nu$ is defined to make the upper triangle commutative. The
  desired equality means that the lower triangle also commutes. To
  prove it, we first claim that the mapping $\nu$, sending each
  $v\in(\Om AV)^I$ to %
  $\eta_{\pv V}^I(v)\cdot x$, is a homomorphism of $A$-algebras.
  Indeed, given $v\in(\Om AV)^I$ and $a\in A$, we have
  \begin{displaymath}
    \eta_{\pv V}^I(av)\cdot x 
    =(\eta_{\pv V}^I(a)\eta_{\pv V}^I(v))\cdot x 
    =\eta_{\pv V}^I(a)\cdot(\eta_{\pv V}^I(v)\cdot x) 
    =a_{\Om XV^u}(\eta_{\pv V}^I(v)\cdot x)
  \end{displaymath}
  Since $\varphi_{\pv V}$ is also a continuous homomorphism of
  $A$-algebras and $\nu(\varphi_{\pv V}(x))=x$, we deduce that 
  Diagram \eqref{eq:iso2-2} commutes.
\end{proof}

The following is an easy consequence of the proof of
Proposition~\ref{p:iso2}.

\begin{Cor}
  \label{c:iso1}
  The mapping $\eta_\pv V$ is an isomorphism.
\end{Cor}

\begin{proof}
  From the commutativity of Diagram \eqref{eq:iso2-2}, we see that
  $\nu$ is an isomorphism of $A$-algebras, namely %
  $\nu=\varphi_{\pv V}^{-1}$. It follows that so is $\eta_{\pv V}^I$
  and, therefore also $\eta_{\pv V}$.
\end{proof}

Proposition~\ref{p:iso2} is also instrumental in establishing the
following key result for our purposes.

\begin{Lemma}
  \label{l:fi}
  Let $\theta$ be a closed congruence on~\Om AV and $\theta'$ its
  extension to the semigroup $(\Om AV)^I$ that is obtained by making 1
  a singleton class. In view of Proposition~\ref{p:iso2}, that
  determines an equivalence relation $\tilde{\theta}$ on $\Om xV^u$.
  The relation $\tilde{\theta}$ is a fully invariant closed congruence
  on~$\Om xV^u$.
\end{Lemma}
  
\begin{proof}
  We first show that $\tilde{\theta}$ is a congruence. Suppose that
  $u,v\in\Om xV^u$ are $\tilde{\theta}$-equivalent and let
  $\varphi_{\pv V}:\Om xV^u\to(\Om AV)^I$ be the isomorphism of
  $A$-algebras given by Proposition~\ref{p:iso2}. Then, the pair
  $(\varphi_{\pv V}(u),\varphi_{\pv V}(v))$ belongs to~$\theta'$ and
  so does %
  $(\varphi_{\pv V}(a\cdot u),\varphi_{\pv V}(a\cdot v)) %
  =(a\varphi_{\pv V}(u),a\varphi_{\pv V}(v))$ for every $a\in A$.
  Hence, $(a\cdot u,a\cdot v)$ belongs to $\tilde{\theta}$.
    
  Next, we verify that $\tilde{\theta}$ is fully invariant. Given a
  continuous endomorphism $\psi$ of~$\Om xV^u$ and a pair
  $(u,v)\in\tilde{\theta}$, if one of $u$ and $v$ is~$x$ then so is the
  other, and so obviously $(\psi(u),\psi(v))\in\tilde{\theta}$. So, we
  may assume that $u,v\ne x$. Let $w=\psi(x)$; we may assume that
  $w\ne x$. By~Proposition~\ref{p:iso2}, we have
  \begin{align*}
    \psi(u)
    &=\psi(\eta_\pv V(\varphi_{\pv V}(u))\cdot x) %
      =\eta_\pv V(\varphi_{\pv V}(u))\cdot w
      =\eta_\pv V(\varphi_{\pv V}(u)\varphi_{\pv V}(w))\cdot x\\
    &\mathrel{\tilde{\theta}} %
      \eta_\pv V(\varphi_{\pv V}(v)\varphi_{\pv V}(w))\cdot x
      =\eta_\pv V(\varphi_{\pv V}(v))\cdot w
      =\psi(\eta_\pv V(\varphi_{\pv V}(v))\cdot x) = \psi(v).\qed\popQED
  \end{align*}
\end{proof}
  
After all the above preparation, we are ready for our counterexample
to \cite[Conjecture~3.3]{Almeida&Klima:2017a} for unary algebras.

\begin{Thm}
  \label{t:ce}
  There is a pseudovariety \pv U of unary algebras and a closed
  fully invariant congruence $\theta$ on~\Om XU that is not
  profinite.
\end{Thm}

\begin{proof}
  Let $A=\{a,b\}$ and let $\theta$ be the equivalence relation on \Om
  AK of Section~\ref{sec:non-profinite-closed}, namely the one that
  identifies two infinite words if and only if they are equal or
  constitute a pair of the form $(wab^\omega,wba^\omega)$, with $w$ a
  finite word; finite words form singleton classes. Note that $\theta$
  is a closed congruence on~\Om AK. Hence, the associated relation
  $\tilde{\theta}$ on~$\Om xK^u$ is a fully invariant closed
  congruence by Lemma~\ref{l:fi}. Note that the infinite words form a
  closed subspace of~\Om AK homeomorphic with the Cantor set (it is
  the perfect kernel of~\Om AK, cf.~\cite[Section~6.B]{Kechris:1995}).
  The relation $\theta$ corresponds to the kernel of the Cantor
  function. Hence, $\Om xK^u/\tilde{\theta}$ is not profinite (it
  contains a closed subspace homeomorphic with the interval $[0,1]$,
  which is not zero-dimensional) thereby showing that $\tilde{\theta}$ is
  not a profinite congruence.
\end{proof}

\section{From profinite unary algebras to profinite monoids}
\label{sec:functor-mu}

In the previous section we constructed, for a given profinite
$A$-algebra $U$, the profinite semigroup $\hat{\Cl S}_U$. Now we
exhibit an alternative construction, with the minor modification that
the resulting structure is a profinite monoid.

Recall that $A$ is a fixed unary signature.
Terms over one single variable $x$ may be written as $w(x)$ with $w\in
A^*$.
Every term $w(x)$ induces, for a given $A$-algebra $U$, the
transformation $\alpha_U(w)\in \Cl S_U$, with one exception, where $w$
is the empty word, the considered term is $x$, and the resulting
transformation is the identity which 
is not
a member of $\Cl S_U$ as
defined in Section~\ref{sec:unary}. For that reason, we add the
identity transformation of~$U$ to $\Cl S_U$, if it is not already
present, and denote the resulting monoid $\Cl M_U$. Moreover, the
considered mapping $\pi : T_A(x) \rightarrow \Cl M_U$ is a monoid
homomorphism if we consider the operation of substitution in $
T_A(x)$; more formally, we definite the operation $\cdot $ on $T_A(x)$
by the rule %
$w(x)\cdot u(x) = wu(x)$.

To imitate the previous construction in the case of profinite
$A$-algebras and avoid working with inverse limits, we consider a
pseudovariety \pv V of $A$-algebras, the free pro-\pv V algebra $\Om
xV$ over the set $\{x\}$, and look at the transformations that are
determined by members of $\Om xV$. We claim that $\Om xV$ is equipped
with the binary operation of substitution generalizing that for
$T_A(x)$. Indeed, by Theorem~\ref{t:evaluation-continuity}, the
evaluation mapping
\begin{align*}
  \Om xV\times\Om xV &\to \Om xV \\
  (u,v) &\mapsto \sigma_v^{\pv V}(u)
\end{align*}
is continuous where, in the notation of Section~\ref{sec:prelims},
$\sigma_v^{\pv V}=\widehat{\varphi_v}$ for the mapping %
$\varphi_v:\{x\}\to\Om xV$ sending $x$ to $v$. %
In case $\pv V=\pv{Un}$, we drop the superscript in the notation
$\sigma_v^{\pv V}$. 

For each pair $u,v\in  \Om xV$, we put $v\cdot u= \sigma_u^{\pv V}(v)$.
In this way we define a continuous multiplication on $\Om xV$. 
We claim that this binary operation is also associative. Indeed, for 
$u,v,w\in \Om xV$ we have 
$$(w\cdot v)\cdot u=\sigma^{\pv V}_{u}(w\cdot v)=
\sigma_u^{\pv V}(\sigma_v^{\pv V} (w)),$$ and 
$$w\cdot (v\cdot u)=\sigma_{v\cdot u}^{\pv V}(w)= 
\sigma^{\pv V}_{\sigma_u^{\pv V} (v)}(w) .$$
Hence, the associativity of the operation $\cdot$ 
is equivalent to stating that
\begin{equation}
  \label{eq:sigma}
  \sigma_{\sigma^{\pv V}_u(v)}^{\pv V}
  =\sigma_{v\cdot u}^{\pv V}
  =\sigma_u^{\pv V}\circ\sigma_v^{\pv V},
\end{equation}
which follows from the commutativity of the following
diagram:
\begin{displaymath}
  \xymatrix{
    \{x\}
    \ar[rr]^\iota
    \ar[rrd]_{\varphi_v}
    \ar@/_1cm/[rrdd]_{\varphi_{v\cdot u}}
    &&
    **[r]{\Om xV}
    \ar[d]^{\sigma_v^{\pv V}}
    \ar@/^1cm/[dd]^{\sigma_{v\cdot u}^{\pv V}}
    \\
    &
    \{x\}
    \ar[r]^\iota
    \ar[rd]_{\varphi_u}
    &
    **[r]{\Om xV}
    \ar[d]^{\sigma_u^{\pv V}}
    \\
    &&
    **[r]{\Om xV .}
  }
\end{displaymath}
Since \Om xV is a topological monoid under the operation $\cdot$, with
neutral element~$x$, and a compact zero-dimensional space, by
\cite{Numakura:1957} it is a profinite monoid.

For a profinite $A$-algebra $U$ we consider the set $\Cl C(U)$ of all
continuous transformations from $U$ to $U$. We endow $\Cl C(U)$ with
the compact-open topology, with subbase of open sets formed by the
subsets of the form
\begin{displaymath}
  [K,O]=\{f\in \Cl C(U) : f(K)\subseteq O\}
\end{displaymath}
where $K$ is compact and $O$ is open. Since the profinite space $U$ is
compact (and Hausdorff), we may choose as subbasic open sets only
those $[K,O]$ where $O$~is clopen: indeed, $O$ is the union of a
family $(O_i)_{i\in I}$ of clopen sets and, from the compactness of
$K$ it follows that
\begin{displaymath}
  [K,O]
  =\bigcup_{F\text{ finite }\subseteq I}[K,\bigcup_{i\in F}O_i].
\end{displaymath}
Note that this smaller subbase of open sets actually consists of clopen
sets:
\begin{displaymath}
  \Cl C(U)\setminus[K,O]
  =\{f\in\Cl U:f(K)\setminus O\ne\emptyset\}
  =\bigcup_{u\in K}[\{u\}, U \setminus O].
\end{displaymath}
It is also a basic observation from general topology that $\Cl C(U)$
is Hausdorff, even under much weaker hypotheses
\cite[Theorem~43.4]{Willard:1970}. Moreover, since $\Cl C(U)$ is
closed under composition, it is a monoid. Notice that composition in
$\Cl C(U)$ is continuous with respect to the compact-open topology ---
more generally, this holds whenever $U$ is locally compact, see
\cite[Theorem 3.4.2]{Engelking:1989}.

Let \pv V be a pseudovariety of $A$-algebras and $U$ a pro-\pv V
algebra. We are now ready to state basic properties of the mapping %
$\tau_U^{\pv V} : \Om xV \rightarrow \Cl C(U)$ defined by the rule
$\tau_U^{\pv V}(w)=w_U$; this is well defined by
Corollary~\ref{c:implops-continuous}.

\begin{Lemma}
  \label{l:closure-of-MU}
  Let \pv V be a pseudovariety of $A$-algebras, $U$ be a pro-\pv V
  algebra, and %
  $\tau_U^{\pv V} : \Om xV \rightarrow \Cl C(U)$ be the mapping
  defined above. Then
  \begin{enumerate}[(i)]
  \item\label{item:closure-of-MU-1} $\tau_U^{\pv V}$ is a continuous
    homomorphism of topological monoids;
  \item\label{item:closure-of-MU-2} $\tau_U^{\pv V}$ is a closed mapping;
  \item\label{item:closure-of-MU-3} the image $\mathrm{Im}(\tau_U^{\pv V})$
    is the closure of $\Cl M_U$ in $\Cl C(U)$.
  \end{enumerate}
\end{Lemma}

\begin{proof}
  We first show that $\tau_U^{\pv V}$ is a monoid homomorphism. So, we
  want to show that $\tau_U^{\pv V}(u\cdot v)$ and $\tau_U^{\pv
    V}(u)\circ \tau_U^{\pv V}(v)$ are the same transformation of~$U$
  for every $u,v\in\Om xV$. We choose an arbitrary $w\in U$ and check
  that %
  $(u\cdot v)_U(w)=u_U\bigl(v_U(w)\bigr)$. Let $\psi$ be the unique
  continuous homomorphism $\Om xV\to U$ that maps $x$ to~$w$. Consider
  the following diagram:
  \begin{displaymath}
    \xymatrix{
      \Om xV
      \ar[d]^\psi
      \ar[r]_{v_{\Om xV}}
      \ar@/^5mm/[rr]^{(u\cdot v)_{\Om xV}}
      &
      \Om xV
      \ar[d]^\psi
      \ar[r]_{u_{\Om xV}}
      &
      \Om xV
      \ar[d]^\psi
      \\
      U
      \ar[r]^{v_U}
      \ar@/_5mm/[rr]_{(u.v)_U}
      &
      U
      \ar[r]^{u_U}
      &
      **[r]{U.}
    }
  \end{displaymath}
  By Corollary~\ref{c:implops}, the two inner squares and the outer
  square commute while the upper triangle commutes
  by~\eqref{eq:sigma}. Hence, we have
  \begin{align*}
    (u\cdot v)_U(w)
    &=(u\cdot v)_U(\psi(x))
      =\psi\bigl((u\cdot v)_{\Om xV}(x)\bigr)
      =\psi\bigl(u_{\Om xV}(v_{\Om xV}(x))\bigr)
    \\
    &=u_U\bigl(\psi(v_{\Om xV}(x))\bigr)
      =u_U\bigl(v_U(\psi(x))\bigr)
      =u_U\bigl(v_U(w)\bigr).
  \end{align*}
  Continuity of $\tau_U^{\pv V}$ follows from
  \cite[Theorem~3.4.7]{Engelking:1989}. This
  proves~(\ref{item:closure-of-MU-1}).
  
  Since $\tau_U^{\pv V}$ is a continuous mapping from the compact
  space $\Om xV$ to the Hausdorff space $\Cl C(U)$, it follows that
  $\tau_U^{\pv V}$ is a closed mapping, which
  gives~(\ref{item:closure-of-MU-2}). From that we also get that
  $\mathrm{Im}(\tau_U^{\pv V})$ is a closed and compact subspace of
  $\Cl C(U)$. We may see $T_{A}(x)$ as a dense subset of $\Om xV$.
  Hence, $\Cl M_U\subseteq \mathrm{Im}(\tau_U^{\pv V})$ and every
  element of $\mathrm{Im}(\tau_U^{\pv V})$ is a limit of some net
  consisting of elements of~$\Cl M_U$.
\end{proof}

We denote the topological monoid $\mathrm{Im}(\tau_U^{\pv V})$ by
$\mu_{\pv V}(U)$. Notice that $\mu_{\pv{Un}}(U)=\Cl M_U$ whenever $U$
is a finite $A$-algebra, because $\Cl C(U)$ is a discrete finite space
in this case.

If $\varphi : U \rightarrow V$ is a continuous homomorphism between
pro-\pv V algebras and $w\in\Om xV$ is an arbitrary element, then the
following diagram commutes by Corollary~\ref{c:implops}:
\begin{displaymath}
  \xymatrix{
    U  \ar[r]^{w_U}
    \ar[d]^{\varphi}
    & U \ar[d]^{\varphi} \\
    V  \ar[r]^{w_V}
    & **[r]{V.} }
\end{displaymath}
Now if, in addition, the mapping $\varphi : U \rightarrow V$ is
surjective, then we may define %
$\reallywidecheck\varphi: \mu_{\pv V}(U) \rightarrow \mu_{\pv V}(V)$
by $\reallywidecheck\varphi (w_U)=w_V$. To see that this definition is
correct, let $w_U=w'_U$ for a pair of elements $w,w'\in \Om xV$. Then
for every $v\in V$ we may choose $u\in U$ such that $\varphi(u)=v$.
Now we have
$w_V(v)=w_V\bigl(\varphi(u)\bigr)=\varphi\bigl(w_U(u)\bigr)$ by the
commutativity of the above diagram. Since $w_U=w'_U$, we get
$\varphi\bigl(w_U(u)\bigr)=\varphi\bigl(w'_U(u)\bigr)=
w'_V\bigl(\varphi(u)\bigr)=w'_V(v)$. Altogether we deduce that
$w_V=w'_V$. Note that $\reallywidecheck\varphi$ is the unique mapping
such that the following diagram commutes:
\begin{displaymath}
  \xymatrix@C=1mm{
    &
    \Om xV
    \ar[ld]_{\tau_U^{\pv V}}
    \ar[rd]^{\tau_V^{\pv V}}
    &
    \\
    \mu_{\pv V}(U)
    \ar[rr]^{\reallywidecheck\varphi}
    &&
    \,\mu_{\pv V}(V).
  }
\end{displaymath}

\begin{Lemma}
  \label{l:onto-homomorphisms}
  If $\varphi : U \rightarrow V$ is an onto continuous homomorphism
  between pro-\pv V algebras, then %
  $\reallywidecheck\varphi: \mu_{\pv V}(U) \rightarrow \mu_{\pv V}(V)$
  is an onto continuous monoid homomorphism.
\end{Lemma}

\begin{proof}
  Clearly, the mapping $\reallywidecheck\varphi$ is onto and it is a
  continuous monoid homomorphism because so are $\tau_U^{\pv V}$ and
  $\tau_V^{\pv V}$ by Lemma~\ref{l:closure-of-MU}.
\end{proof}

The expected statement follows.

\begin{Prop}
  \label{p:profinite-monoid-mu}
  If $U$ is a pro-$\pv V$ $A$-algebra, then $\mu_{\pv V}(U)$ is a profinite
  monoid.
\end{Prop}

\begin{proof}
  In view of Lemma~\ref{l:closure-of-MU}, $\mu_{\pv V}(U)$ is a compact
  submonoid of the topological monoid $\Cl C(U)$. Since $\Cl C(U)$ is
  zero-dimensional, $\mu_{\pv V}(U)$ is a topological monoid on a profinite
  space. By~\cite{Numakura:1957}, it follows that $\mu_{\pv V}(U)$ is a
  profinite monoid.
\end{proof}

The following is a further application of Lemma~\ref{l:closure-of-MU}.

\begin{Prop}
  \label{p:monoid-profiniteness-for-free-A-algebra}
  For a pseudovariety \pv V of $A$-algebras, the profinite monoids
  $(\Om xV,\cdot)$ and $(\mu_{\pv V} (\Om XV),\circ)$ are isomorphic.
\end{Prop}

\begin{proof}
  We already know that $(\Om xV,\cdot)$ is 
  a profinite monoid and that %
  $\tau_U^{\pv V} : \Om xV \rightarrow \mu_{\pv V} (\Om xV)$ is an onto
  continuous monoid homomorphism by Lemma~\ref{l:closure-of-MU}. We
  also know that $\tau_U^{\pv V}$ is a closed mapping. If we show that
  $\tau_U^{\pv V}$ is injective, then it is a bijective continuous
  closed mapping, and therefore it is a homeomorphism. And as a
  bijective monoid homomorphism it is also a monoid isomorphism. Thus,
  the whole statement is proved if we show that $\tau_U^{\pv V}$ is
  injective.
  
  For a pair of distinct elements $w,w'\in \Om xV$, we want to show
  that the transformations $w_U,w'_U\in\mu_{\pv V}(U)$ are distinct.
  Since $U=\Om xV$ is freely generated by $x$, we must
  show that
  $w_U(x)\ne w'_U(x)$. However, we have $w_U(x)=\sigma_x^{\pv
  V}(w)=w$.
  Hence, we get $w'_U(x)=w'\ne w=w_U(x)$.
\end{proof}

To summarize, for any pro-\pv V algebra $U$, we construct the
profinite monoid $\mu_{\pv V}(U)$. Moreover, in view of
Lemma~\ref{l:onto-homomorphisms}, this construction respects onto
homomorphisms. So, it is natural to ask whether the correspondence
$\mu_{\pv V}$ works as a functor from the point of view of category
theory. We specify appropriate categories later, now we just observe,
that $\mu_{\pv V}$ respects composition of surjective homomorphisms.

\begin{Lemma}
 \label{l:composition-vs-hat}
 Let $\varphi :U \rightarrow V$ and $\psi : V \rightarrow W$ are onto
 continuous homomorphisms between pro-\pv V algebras. Then
 the equality
 $\reallywidecheck{\psi\circ\varphi}=\reallywidecheck\psi\circ
 \reallywidecheck\varphi$ holds.
\end{Lemma}

\begin{proof}
  Let $w\in\Om xV$ be an arbitrary element. Then we may see that
  \begin{displaymath}
    (\reallywidecheck\psi\circ \reallywidecheck\varphi) (w_U)
    =\reallywidecheck\psi \bigl(\reallywidecheck\varphi(w_U)\bigr)
    =\reallywidecheck\psi (w_V)
    =w_W .
  \end{displaymath}
  For composite continuous homomorphism %
  $\psi\circ \varphi : U \rightarrow W$, we also have
  $\reallywidecheck{\psi\circ\varphi}(w_U)=w_W$.
\end{proof}

Although, for a pro-\pv V algebra~$U$, we have defined $\mu_{\pv
  V}(U)$ using the pseudovariety \pv V, it turns out that this
profinite monoid depends only on~$U$.

\begin{Prop}
  \label{p:V-independence}
  Let $U$ be a pro-\pv V algebra. Then the equality %
  $\mu_{\pv V}(U)=\mu_{\pv{Un}(U)}$ holds.
\end{Prop}

\begin{proof}
  Consider the following special case of
  Diagram~\eqref{eq:evaluation-diagram}:
  \begin{displaymath}
    \xymatrix{
      \Om x{Un}\times U
      \ar[r]^(.65){\varepsilon_{\pv{Un}}^U}
      \ar[d]^{\pi_{\pv V}\times\mathrm{id}_U}
      &
      U
      \ar[d]^{\mathrm{id}_U}
      \\
      \Om xV\times U
      \ar[r]^(.65){\varepsilon_{\pv V}^U}
      &
      U
    }
  \end{displaymath}
  Since the diagram commutes, as was established in the proof of
  Theorem~\ref{t:evaluation-continuity}, we obtain the following chain
  of equalities for $w\in\Om x{Un}$ and $u\in U$:
  \begin{displaymath}
    w_U(u)
    =\varepsilon_{\pv{Un}}^U(w,u)
    =\varepsilon_{\pv V}^U\bigl(\pi_{\pv V}(w),u\bigr)
    =\bigl(\pi_{\pv V}(w)\bigr)_U(u).
  \end{displaymath}
  This shows that $\mu_{\pv{Un}}(U)=\mu_{\pv V}(U)$.
\end{proof}

In view of Proposition~\ref{p:V-independence}, from hereon we drop the
subscript \pv V in the notation $\mu_{\pv V}$.

To conclude this section, we compare the construction of the profinite
monoid $\mu (U)$ with the construction of the profinite monoid given
via inverse limits in Section~\ref{sec:unary}. Let $U=\varprojlim U_i$
be an inverse limit of finite $A$-algebras with onto connecting
homomorphisms $\varphi_{ij}:U_i\to U_j$ ($i\ge j$). There is an
associated inverse system of $A$-generated finite monoids $\Cl
M_{U_i}=\mu (U_i)$ with connecting homomorphisms
$\reallywidecheck\varphi_{ij}:\Cl M_{U_i}\to\Cl M_{U_j}$ ($i\ge j$).
Since, for every $i$, we have an onto continuous homomorphism %
$\varphi_i : U \rightarrow U_i$, we also have the corresponding
continuous homomorphism of monoids %
$\reallywidecheck{\varphi_i} :\mu(U) \rightarrow \mu(U_i)$. By
Lemma~\ref{l:composition-vs-hat}, all homomorphisms between $\mu (U)$
and the $\mu (U_i)$'s compose in the right way. To see that $\mu (U)$
is the inverse limit $\varprojlim \mu (U_i)$, we only need to show
that, for every two distinct elements $w_U,w'_U\in \mu (U)$, there is
an index $i$ such that %
$\reallywidecheck{\varphi_i} (w_U) %
\not=\reallywidecheck{\varphi_i} (w'_U)$. %
We show that in the next result.

\begin{Prop}
  \label{p:alternative-construction}
  Let $U$ be a profinite $A$-algebra which is an inverse limit of
  finite $A$-algebras $(U_i)_{i}$ as above. Then $\mu (U)$ is an
  inverse limit of the corresponding system of finite monoids
  $\bigl(\mu(U_i)\bigr)_{i}$.
\end{Prop}

\begin{proof}
  By the comment before the statement of the proposition, we have to
  show just one detail. Let $w,w'\in\Om x{Un}$ be such that $w_U\ne
  w'_U$. Then there is $u\in U$ such that $w_U(u)\ne w'_U(u)$. Since
  $U$ is an inverse limit of $A$-algebras $(U_i)_{i}$, there is $i$
  such that $\varphi_i(w_U(u))\ne \varphi_i(w'_U(u))$. However, from
  Corollary~\ref{c:implops}, we know that
  $\varphi_i\bigl(w_U(u)\bigr)=w_{U_i}\bigl(\varphi_i(u)\bigr)=
  \reallywidecheck{\varphi_i}(w_U)\bigl(\varphi_i(u)\bigr)$ and,
  similarly, $\varphi_i\bigl(w'_U(u)\bigr)=
  \reallywidecheck{\varphi_i}(w'_U)\bigl(\varphi_i(u)\bigr)$. It
  follows that
  $\reallywidecheck{\varphi_i}(w_U)\ne\reallywidecheck{\varphi_i}(w'_U)$
  because they map $\varphi(u)\in U_i$ to distinct points.
\end{proof}

\section{From profinite monoids to profinite algebras}
\label{sec:from-monoid-to-algebra}

In Secion~\ref{sec:unary} we constructed, for a given semigroup, a
certain $A$-algebra. We generalize that construction for an arbitrary
signature. However, the unary case is the one we study in most detail.
Once again, we work with monoids instead of semigroups in this
section.

Let $\Sigma$ be an arbitrary signature, let $M$ be a monoid, and let
$\gamma : \Sigma \rightarrow M$ be a mapping. We denote by
$\kappa_\gamma (M)$ the following $\Sigma$-algebra: the domain is $M$,
and for every arity $n$, we consider the evaluation mapping %
$E_n: \Sigma_n \times M^n \rightarrow M$ given by
\begin{displaymath}
  E_n(w,m_1,\dots , m_n)=\gamma(w)\cdot m_1 \cdot \dots \cdot m_n .
\end{displaymath}
This construction is robust in the sense of the following lemmas.

\begin{Lemma}
\label{l:kappa-preserves-homomorphisms}
Let $\varphi : M \rightarrow N$ be a monoid homomorphism. Let
$\kappa_\gamma (M)$ be the $\Sigma$-algebra defined for a fixed
mapping $\gamma : \Sigma \rightarrow M$ and %
$\kappa_{\varphi\circ \gamma} (N)$ be the $\Sigma$-algebra defined for
the mapping $\varphi\circ \gamma : \Sigma \rightarrow N$. Then %
$\varphi : \kappa_\gamma (M) \rightarrow \kappa_{\varphi\circ\gamma}
(N)$ is a homomorphism of $\Sigma$-algebras.
\end{Lemma}

\begin{proof}
  The following diagram commutes, for every arity $n$:
  \begin{displaymath}
    \xymatrix{
      \Sigma_n\times M^n  \ar[r]^(.6){E_n^M}
      \ar[d]^{\mathrm{id}_{\Sigma_n}\times \varphi^n} 
      & M \ar[d]^{\varphi} \\
      \Sigma_n\times N^n  \ar[r]^(.6){E_n^N}
      & **[r]{N.} }
  \end{displaymath}
  It remains to observe that the commutativity of the diagram means
  that $\varphi$ is a homomorphism from the $\Sigma$-algebra
  $\kappa_\gamma (M)$ to the $\Sigma$-algebra
  $\kappa_{\varphi\circ\gamma} (N)$.
\end{proof}

\begin{Lemma}
  Let $\kappa_\gamma (M)$ be the $\Sigma$-algebra defined for a fixed
  monoid $M$ and a mapping $\gamma : \Sigma \rightarrow M$. Then the
  following hold: 
  \begin{enumerate}[(i)]
  \item If $M$ is a topological monoid, then $\kappa_\gamma (M)$
    is a topological $\Sigma$-algebra.
  \item If $M$ is a profinite monoid, then $\kappa_\gamma (M)$ is a
    profinite $\Sigma$-algebra.
  \end{enumerate}
\end{Lemma}

\begin{proof}
  The first statement is clear, since $E_n$ is continuous whenever the
  multiplication on $M$ is continuous.
  
  Assume that $M$ is profinite. For a pair of distinct elements
  $m\not=m'$, there is a continuous homomorphism $\varphi : M
  \rightarrow N$ onto a finite monoid $N$ such that
  $\varphi(m)\not=\varphi(m')$. By
  Lemma~\ref{l:kappa-preserves-homomorphisms} we know that $\varphi$
  is a homomorphism of $\Sigma$-algebras, which is continuous, as
  topologies are kept.
\end{proof}

Now we move our attention to the unary case, that is, the case when
$\Sigma_1=A$, and $\Sigma_n=\emptyset$ for every $n> 1$. If we look at
Section~\ref{sec:unary}, we see that the results deal with
$A$-generated semigroups. In the setting of this section, this
restriction may be expressed as $\gamma(\Sigma)$ generating a dense
submonoid of the profinite monoid~$M$. Moreover, the constructed
profinite $A$-algebra $\kappa_\gamma(M)$ is generated by the element
$1\in M$ (again in the algebraic-topological sense). To put it more
formally, we let $\Sigma_0=\{1\}$ and we deal with profinite
$\Sigma$-algebras which have no proper closed subalgebras.
Then the natural homomorphisms between our structures are
automatically surjective. This restriction fits to observations from
Section~\ref{sec:functor-mu}. In this way we fix categories for which
we want to relate constructions from Section~\ref{sec:functor-mu} and
this section.

Let $\mathbb{PUA}_A^1$ be the category whose objects are profinite
unary $\Sigma$-algebras without proper closed subalgebras and
morphisms are continuous homomorphisms between such algebras. In
particular, every morphism is an onto mapping, because every object
$U$ is a profinite $A$-algebra $U$ generated by the single element
$1^U$, where $1\in\Sigma_0$ is the unique nullary operational symbol.
Notice that $\mathbb{PUA}_A^1$ is a thin category, that is, for every
two objects $U$ and $V$ there is at most one morphism from $U$ to $V$.

Furthermore, let $\mathbb{PM}_A$ be the category defined in the
following way. The objects are $A$-generated profinite monoids, or
more formally, mappings $\gamma : A \rightarrow M$, where $M$ is a
profinite monoid such that the closed submonoid generated by
$\mathrm{Im}(\gamma)$ is $M$. And morphisms are surjective continuous
monoid homomorphisms or, more formally, $\varphi$ is a morphism from
$\gamma : A \rightarrow M$ to $\beta :A\rightarrow N$ if $\varphi : M
\rightarrow N$ is a continuous monoid homomorphism and $\varphi \circ\
\gamma =\beta$. Also the category $\mathbb{PM}_A$ is thin and all
morphisms are surjective mappings.

Recall, that for an object $U$ in $\mathbb{PUA}_A^1$ we have
constructed the profinite monoid $\mu (U)$. This is not formally an
object in the category $\mathbb{PM}_A$, since we need to fix a mapping
$\gamma : A \rightarrow \mu (U)$. However, this mapping is canonically
determined in Lemma~\ref{l:closure-of-MU}. Indeed, since $\Cl M_U$ is
generated (as a monoid) by transformations given by $a\in A$, we have
$\gamma (a)=(ax)_U\in \mu(U)$. The same lemma ensures that this
mapping $\gamma : A \rightarrow \mu(U)$ is an object in the category
$\mathbb{PM}_A$. We denote this $\gamma$ by $\Gamma(U)$. Moreover,
if there is a morphism $\varphi :U \rightarrow V$ in the category
$\mathbb{PUA}_A^1$, then $\reallywidecheck\varphi$ is a morphism
between the objects $\Gamma(U)$ and $\Gamma(V)$ by
Lemma~\ref{l:onto-homomorphisms} and the considerations before that
lemma. Thus, we denote $\Gamma (\varphi) =\reallywidecheck\varphi$ and
we see that $\Gamma : \mathbb{PUA}_A^1 \rightarrow \mathbb{PM}_A$
is a functor of these thin categories by
Lemma~\ref{l:composition-vs-hat}.

Now, for a given $\gamma : A \rightarrow M$ in $\mathbb{PM}_A$, the
$A$-algebra $\kappa_\gamma (M)$ is profinite and generated by the
element $1\in M$. Thus, we may denote $\Psi(\gamma)=\kappa_\gamma
(M)$, which is an object in $\mathbb{PUA}_A^1$. If $\varphi$ is a
morphism in $\mathbb{PM}_A$, in particular it is a surjective
homomorphism $\varphi : M \rightarrow N$, then by
Lemma~\ref{l:kappa-preserves-homomorphisms} the mapping $\varphi :
\Psi(\gamma) \rightarrow \Psi (\varphi\circ\gamma)$ is a continuous
homomorphism of $A$-algebras. If we denote it as $\Psi(\varphi)$, then
we get a functor $\Psi : \mathbb{PM}_A \rightarrow \mathbb{PUA}_A^1$
between these thin categories. The only detail that might require some
checking is that $\Psi(\varphi)\circ
\Psi(\varphi')=\Psi(\varphi\circ\varphi')$ for a pair of morphisms
$\varphi : \gamma \rightarrow (\varphi\circ\gamma)$ and $\varphi' :
(\varphi\circ\gamma) \rightarrow (\varphi'\circ\varphi\circ\gamma)$ in
the category $\mathbb{PM}_A$. But this is trivial since
$\Psi(\psi)=\psi$ for every object $\psi$ in $\mathbb{PM}_A$.

We may interpret the monoid analog of Lemma~\ref{l:Cayley-profinite}
with respect to Proposition~\ref{p:alternative-construction} as the
following statement.

\begin{Lemma}
  \label{l:isomorphism-via-functors}
  For every object $\gamma$ from $\mathbb{PM}_A$, we have 
  $\Gamma(\Psi(\gamma))\cong \gamma$. \qed 
\end{Lemma}

Composing the functors in the reverse order leads to the following
result.

\begin{Lemma}
  \label{l:psi-after-gamma}
  For every object $U$ from $\mathbb{PUA}_A^1$,
  there exists a surjective continuous homomorphism
  from $\Psi(\Gamma(U))$ to $U$.
\end{Lemma}

\begin{proof}
  Recall that $\Psi(\Gamma(U))=\kappa_{\Gamma(U)}(\mu (U))$ is the
  $A$-algebra with domain $\mu (U)=\{w_U : w\in\Om{x}{Un}\}$, where 
  $U$ is an $A$-algebra generated by the element $1^U$.
  We define the mapping $\beta : \mu (U) \rightarrow U$
  by the rule $\beta(w_U)=w_U(1^U)$ for every $w\in \Om{x}{Un}$. 
  This definition is correct as $w_U=w'_U$ implies $w_U(1^U)=w'_U(1^U)$.
  Moreover, $\beta$ is surjective because $U$ is generated by $1^U$ as 
  a topological algebra.
  
  By the definition of $w_U$, we have
  $\beta(w_U)=\varphi(w)$, where $\varphi :
  \Om{x}{Un} \rightarrow U$ is the unique continuous homomorphism of
  $\Sigma$-algebras.
  From the definition of $\beta$ we obtain that the
  following diagram commutes:
  \begin{displaymath}
    \xymatrix@C=1mm{
      \Om x{Un}
      \ar[rr]^{\tau_U}
      \ar[rd]_\varphi
      &&
      \mu(U)
      \ar[ld]^\beta
      \\
      &
      **[r]{U.}
      &
    }
  \end{displaymath}
  Since $\tau_U$ is a continuous homomorphism of $A$-generated
  topological monoids, it is also a homomorphism of $A$-algebras.
  Since so is~$\varphi$, it follows that $\beta$ is again a
  homomorphism of $A$-algebras. Since both mappings $\tau_U$ and
  $\varphi$ are continuous, we also conclude that $\beta$ is
  continuous: if we consider a closed subset $C$ of $U$, then
  $\beta^{-1}(C)=\tau_U(\varphi^{-1}(C))$ where $\varphi$ is
  continuous and $\tau_U$ is closed.
\end{proof}

We use the previous observations to establish the following result.

\begin{Thm} 
  \label{t:left-adjoint}
  The functor $\Psi: \mathbb{PM}_A  \rightarrow \mathbb{PUA}_A^1$
  is left-adjoint to the functor
  $\Gamma : \mathbb{PUA}_A^1 \rightarrow \mathbb{PM}_A$. 
\end{Thm}

\begin{proof}
  To show that the functors are adjoint, we have to show that,
  for every $\gamma\in \mathbb{PM}_A$ and $U\in\mathbb{PUA}_A^1$, we have 
  $$\operatorname{hom}_{\mathbb{PUA}_A^1}\big(\Psi(\gamma),U\big)\cong
  \operatorname{hom}_{\mathbb{PM}_A}\big(\gamma,\Gamma(U)\big) ,$$ where
  the family of bijections is natural. 
  Since the categories are thin, the naturality is trivial whenever 
  we show that there exists a morphism from  $\Psi(\gamma)$ to $U$ if and only 
  if there exists a morphism from $\gamma$ to $\Gamma(U)$.
 
  The implication from left to right is clear because it is enough to
  apply the functor $\Gamma$ and use
  Lemma~\ref{l:isomorphism-via-functors}. Similarly, if we have a
  morphism $\alpha :\gamma \rightarrow \Gamma(U)$, then we may use the
  functor $\Psi$ and we get $\Psi(\alpha) :\Psi(\gamma) \rightarrow
  \Psi(\Gamma(U))$. If we compose this morphism with the morphism
  given by Lemma~\ref{l:psi-after-gamma}, then we
  get an appropriate morphism from $\Psi(\gamma)$ to $U$.
\end{proof}

The main application of the previous result is that the functor
$\Gamma$ is \emph{continuous} in the sense that it preserves limits
while $\Psi$ is \emph{cocontinuous}, that is, it preserves colimits.
We already showed in Proposition~\ref{p:alternative-construction},
that $\Gamma$ preserves inverse limits, which are a special case of
limits with directed diagrams. Thus
Proposition~\ref{p:alternative-construction} may be viewed as a
consequence of Theorem~\ref{t:left-adjoint}. However, in fact, we used
Proposition~\ref{p:alternative-construction} in the proof of
Theorem~\ref{t:left-adjoint} via
Lemma~\ref{l:isomorphism-via-functors}. %
It should also be pointed out that the discrete analog of
Theorem~\ref{t:left-adjoint} has been previously considered
in~\cite{Planting:2013}.

\section{The Polish representation}
\label{sec:Polish}

Recall the \emph{Polish notation} for a term, which
basically drops from the usual notation all parentheses and commas.
For instance, for the term
\begin{displaymath}
  u(v(x_1,u(x_2,x_1),x_3),u(x_3,x_2)),
\end{displaymath}
where $u\in\Sigma_2$ and $v\in\Sigma_3$, the Polish notation is the
word $uvx_1ux_2x_1x_3ux_3x_2$ in the alphabet $X\cup\Sigma$. The idea
is that the arity of the operation symbols allows one to uniquely
recover each term from its Polish notation. The Polish notations of
terms thus live in the free monoid $(X\cup\Sigma)^*$.

Let \pv V be a pseudovariety of monoids. The mapping
$\gamma:\Sigma\to\Om{X\cup\Sigma}V$ obtained by restriction of the
natural generating mapping $X\cup\Sigma\to\Om{X\cup\Sigma}V$
determines on the set $M=\Om{X\cup\Sigma}V$ a structure of profinite $\Sigma$-algebra,
namely $\kappa_\gamma(M)$, as argued at the beginning of
Section~\ref{sec:from-monoid-to-algebra}.
The closed subalgebra generated by $X$ is denoted $S_{\pv V}$. Note
that it is a profinite $\Sigma$-algebra.

Following \cite{Almeida&Costa:2015hb}, we say that a profinite
algebra $S$ with generating mapping $\iota:X\to S$ is
\emph{self-free}, or that $S$ is \emph{self-free with basis $X$},
if every continuous mapping $\varphi:X\to S$ induces a continuous
endomorphism $\hat{\varphi}$ of~$S$ such that
$\hat{\varphi}\circ\iota=\varphi$.

\begin{Prop}
  \label{p:self-free}
  For every pseudovariety of monoids \pv V, the profinite $\Sigma$-algebra
  $S_{\pv V}$ is self-free with basis $X$. 
\end{Prop}

\begin{proof}
  Consider the natural generating mapping $\iota:X\to S_{\pv V}$ and
  let $\varphi:X\to S_{\pv V}$ be an arbitrary continuous mapping. We
  may extend $\varphi$ to a continuous function
  $X\cup\Sigma\to\Om{X\cup\Sigma}V$, which we still denote $\varphi$,
  by letting the restriction of $\varphi$ to~$\Sigma$ coincide with
  that of the natural generating mapping. By the universal property
  of~\Om{X\cup\Sigma}V, there is a unique continuous (monoid)
  endomorphism $\hat{\varphi}$ of~$\Om{X\cup\Sigma}V$ such that
  $\hat{\varphi}\circ\iota=\varphi$. Since $\hat{\varphi}$ is the
  identity on~$\Sigma$, $\hat{\varphi}$ is a homomorphism of
  $\Sigma$-algebras. Since $\varphi(X)\subseteq S_{\pv V}$, the
  restriction of~$\hat{\varphi}$ to~$S_{\pv V}$ takes its values
  in~$S_{\pv V}$. Hence, $\varphi$ does extend to a continuous
  endomorphism of~$S_{\pv V}$.
\end{proof}

Combining Proposition~\ref{p:self-free} and
\cite[Theorem~2.16]{Almeida&Costa:2015hb}, we obtain the following
result.

\begin{Cor}
  \label{c:S-V-free}\
  For every pseudovariety \pv V of monoids, discrete space $X$, and
  discrete signature $\Sigma$, there is a pseudovariety of
  $\Sigma$-algebras \Cl V such that the closed subalgebra $S_{\pv V}$
  of~\Om{X\cup\Sigma}V defined above is isomorphic with~$\Om X{\Cl
    V}$.
\end{Cor}

For the pseudovariety \pv M of all finite monoids,
there is a natural (onto) continuous homomorphism $\Phi_X:\Om
X{Fin}_\Sigma\to S_{\pv M}$ of $\Sigma$-algebras, 
namely the only one such that
$\Phi_X(x)=x$ for each $x\in X$. We call it the \emph{Polish
  representation} (of~$\Om X{Fin}_\Sigma$).

The next result shows that $\Phi_X$ is in general not injective so
that even the pseudovariety of finite algebras \Cl V of
Corollary~\ref{c:S-V-free} corresponding to taking $\pv V=\pv M$ is a
proper subpseudovariety of~$\pv{Fin}_\Sigma$.

\begin{Prop}
  \label{p:not-faithful}
  Suppose that $\Sigma$ is a topological signature containing at least
  one symbol $u$ of arity at least~$n\ge2$. Let $X$ be a nonempty
  topological space. Then, the Polish representation $\Phi_X$ is not
  injective.
\end{Prop}

\begin{proof}
  Let $x,y,z$ be three distinct variables, where we only require that
  $x\in X$. Then, writing $x^{\omega+1}$ for the product $xx^\omega$, the
  following equality holds in $\Om{X\cup\Sigma}M$ because
  $\omega$-powers are idempotents:
  \begin{equation}
    \label{eq:not-faithful-1}
    u^\omega(u^\omega x^{\omega+1})^{\omega+1} %
    =(u^\omega x^{\omega+1})^{\omega+1} %
    = u^\omega u^\omega(u^\omega x^{\omega+1})^{\omega+1} x^\omega.
  \end{equation}
  We may consider the $\Sigma$-term
  \begin{displaymath}
    w_k(x,y,z)=\underbrace{u\biggl(u\Bigl(\cdots u\bigl(u}_{k}(x,
    \underbrace{y,\ldots,y}_{n-1}),
    \underbrace{
      \underbrace{z,\ldots,z}_{n-1}\bigr),\ldots\Bigr),
      \underbrace{z,\ldots,z}_{n-1}}_{k-1}
    \biggr).
  \end{displaymath}
  We further let
  \begin{displaymath}
    t_k=w_k(w_k(x,x,x),w_k(x,x,x),w_k(x,x,x)).
  \end{displaymath}
  Note that $\Phi_X(t_k)=u^k(u^kx^{k(n-1)+1})^{k(n-1)+1}$. Thus, if
  $t$ is any accumulation point in~$\Om X{Fin}_\Sigma$ of the sequence
  $(t_{r!})_r$, then the continuity of~$\Phi_X$ yields the equality
  \begin{equation}
    \label{eq:not-faithful-2}
    \Phi_X(t)=u^\omega(u^\omega x^{\omega+1})^{\omega+1}.
  \end{equation}
  Similarly, we may consider the $\Sigma$-term
  $$s_k=w_k(t_k,x,x),$$
  and let $s$ denote an accumulation point in~$\Om X{Fin}_\Sigma$ of
  the sequence $(s_{r!})_r$. Again, continuity of~$\Phi_X$ gives the
  equality
  \begin{equation}
    \label{eq:not-faithful-3}
    \Phi_X(s)=u^\omega u^\omega(u^\omega x^{\omega+1})^{\omega+1}x^\omega.
  \end{equation}
  Combining the equations \eqref{eq:not-faithful-1},
  \eqref{eq:not-faithful-2}, and~\eqref{eq:not-faithful-3}, we
  conclude that, if $\Phi_X$ were injective, then the equality $t=s$
  would hold in $\Om X{Fin}_\Sigma$. 
  To show that this leads to a contradiction, we show that
  there is a continuous homomorphism $\xi:\Om X{Fin}_\Sigma\to A$ into
  a finite $\Sigma$-algebra such that $\xi(s)\ne\xi(t)$.
  
  Let $A=\{a,b\}$ and define on $A$ a structure of $\Sigma$-algebra by
  interpreting each operation in $\Sigma_m$ with $m\ne n$ as the
  constant operation with value $a$ and each operation $v\in\Sigma_n$
  by letting $v_A(a_1,\ldots,a_n)=\widetilde{a_n}$, where
  $\widetilde{a}=b$ and $\widetilde{b}=a$. In this way, $A$ is a
  finite discrete topological algebra. Choose $\xi$ to be any
  continuous homomorphism $\xi:\Om X{Fin}_\Sigma\to A$ such that
  $\xi(x)=a$. In view of the definition of $u_A$, we have
  $\xi\bigl(w_k(x,x,x)\bigr)=\widetilde{a}=b$ and so $\xi(t_k)=\widetilde{b}=a$
  while $\xi(s_k)=\widetilde{a}=b$. Since $\xi$ is continuous, it
  follows that $\xi(t)=a$ and $\xi(s)=b$, which establishes the claim.
\end{proof}

In contrast, free profinite unary algebras are faithfully represented
through the Polish representation, which is yet another result showing
that unary algebras are rather special.

\begin{Thm}
  \label{t:Polish-faithful-unary}
  Let $\Sigma=\Sigma_0\cup\Sigma_1$ be an at most unary topological
  signature and let $X$ be an arbitrary topological space. Then, the
  Polish representation $\Phi_X:\Om X{Fin}_\Sigma\to S_{\pv M}$ 
  is an isomorphism of topological algebras.
\end{Thm}

\begin{proof}
  Since $\Phi_X$ is an onto continuous homomorphism, it suffices to
  show that it is injective. Suppose that $a$ and $b$ are distinct
  elements of $\Om X{Fin}_\Sigma$. Since $\Om X{Fin}_\Sigma$ is
  residually finite, there exists a continuous homomorphism %
  $\varphi:\Om X{Fin}_\Sigma\to F$ into a finite $\Sigma$-algebra such
  that $\varphi(a)\ne\varphi(b)$. Let $M=M(F)\uplus F\uplus\{0\}$ 
  be a disjoint union of finite discrete spaces, so that it is itself
  a finite discrete space.
  We define on $M$ a multiplication as follows:
  \begin{itemize}
  \item for $g,h\in M(F)$, we let $g\cdot h=g\circ h$;
  \item for $g\in M(F)$ and $s\in F$, we put $g\cdot s=g(s)$;
  \item all remaining products are set to be 0.
  \end{itemize}
  It is easy to verify that $M$ is a monoid for the above
  multiplication. Let $\iota:X\to\Om X{Fin}_\Sigma$ be the natural
  generating mapping and consider the evaluation mappings
  $E_0:\Sigma_0\to F$ and $E_1:\Sigma_1\times F\to F$. Note that $E_1$
  induces a continuous mapping $\varepsilon:\Sigma_1\to M(F)$ into the
  discrete monoid $M(F)$ defined by $\varepsilon(u)(s)=E_1(u,s)$ for
  $u\in\Sigma_1$ and $s\in F$. We may now define a continuous mapping
  $\psi:X\cup\Sigma\to M$ by letting $\psi(x)=\varphi(\iota(x))$ for
  $x\in X$, $\psi(c)=E_0(c)$ for $c\in\Sigma_0$, and
  $\psi(u)=\varepsilon(u)$ for $u\in\Sigma_1$. By the universal
  property of the free profinite monoid~$\Om{X\cup\Sigma}M$, 
  $\psi$ induces a continuous
  homomorphism of monoids %
  $\hat{\psi}:\Om{X\cup\Sigma}M\to M$ such that
  $\hat{\psi}\circ\eta=\psi$, where
  $\eta:X\cup\Sigma\to\Om{X\cup\Sigma}M$ is the natural generating
  mapping. The following diagram may help the reader to keep track of
  all these continuous mappings.
  
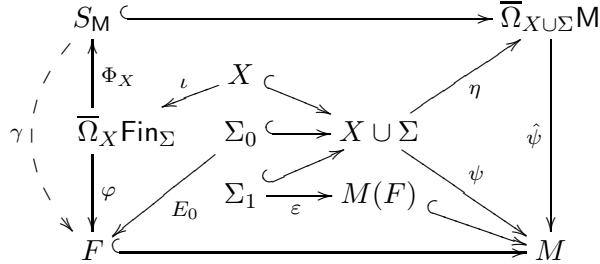
\begin{displaymath}
    \xymatrix@R=2mm{
      S_{\pv M} \ar@{^{(((}->}[rrr] \ar@{-->}@/_8mm/[dddd]_\gamma
      &&& \Om{X\cup\Sigma}M \ar[dddd]_{\hat{\psi}} \\
      & X \ar[ld]_(.35)\iota \ar@{^{(((}->}[rd]
      && \\
      **[r]{\Om X{Fin}_\Sigma} \ar[dd]^\varphi \ar[uu]_{\Phi_X}
      & \Sigma_0 \ar[ldd]^{E_0} \ar@{^{(((}->}[r]
      & X\cup\Sigma \ar[ruu]_\eta \ar[rdd]^\psi
      & \\
      & \Sigma_1 \ar[r]_(.4)\varepsilon \ar@{^{(((}->}[ru]
      & M(F) \ar@{^{(((}->}[rd]
      & \\
      F \ar@{^{(((}->}[rrr]
      && 
      & M
    }
  \end{displaymath}
  
  Let $\gamma$ be the restriction of~$\hat{\psi}$ to $S_M$. Note that
  $\gamma$ takes its values in~$F$: all elements of~$S_{\pv M}$ are of
  the form $wa$, where $w$ belongs to the closed submonoid
  of~$\Om{X\cup\Sigma}M$ generated by~$\eta(\Sigma_1)$ and
  $a\in\overline{\eta(X\cup\Sigma_0)}$; hence, we have
  $\hat{\psi}(wa)=\hat{\psi}(w)\hat{\psi}(a)$, where $\hat{\psi}(w)\in
  M(F)$ and $\hat{\psi}(a)\in F$, resulting in $\hat{\psi}(wa)\in F$
  according to the definition of the multiplication in~$M$. Next, we
  claim that $\gamma$ is a homomorphism of $\Sigma$-algebras. Indeed,
  for $u\in\Sigma_1$ and $v\in S_{\pv M}$, the following equalities
  hold:
  \begin{displaymath}
    \hat{\psi}(u_{S_{\pv M}}(v))
    =\hat{\psi}(\eta(u)v)
    =\hat{\psi}(\eta(u))\hat{\psi}(v)
    =\varepsilon(u)\hat{\psi}(v)
    =u_F(\hat{\psi}(v)).
  \end{displaymath}

  Note that the diagram commutes as $\gamma\circ\Phi_X=\varphi$:
  since, by the above, both sides of the equation are continuous
  homomorphisms of $\Sigma$-algebras, it suffices to check that
  composing them with $\iota$ we obtain an equality and, indeed, for
  $x\in X$, we have
  \begin{displaymath}
    \varphi(\iota(x))
    =\psi(x)
    =\hat{\psi}(\eta(x))
    =\hat{\psi}(\Phi_X(\iota(x)))
    =\gamma(\Phi_X(\iota(x))).
  \end{displaymath}
  Finally, since $\varphi(a)\ne\varphi(b)$, it follows that
  $\Phi_X(a)\ne\Phi_X(b)$, which establishes that $\Phi_X$ is injective.
\end{proof}

\bibliographystyle{amsplain}
\bibliography{sgpabb,ref-sgps}

\end{document}